\documentclass{gtpart}   
\usepackage{amsfonts, amstext, amsmath, amsthm, amscd, amssymb}
\usepackage{color,graphics,graphicx,enumerate, comment,setspace,pinlabel,hyperref,caption}
\usepackage[all]{xy}
\usepackage[paper=a4paper, text={138mm,208mm},centering]{geometry}
\xyoption{import,dvips,rotate}

\title[Double $L$--groups and doubly slice knots]{Double $L$--groups and doubly slice knots}

\author{Patrick Orson} 
\givenname{Patrick}
\surname{Orson}
\address{Department of Mathematics\\University of Durham\\United Kingdom}
\email{patrick.orson@durham.ac.uk}
\urladdr{http://www.maths.dur.ac.uk/~vkdx72/}

\keyword{Doubly slice knots}
\keyword{Blanchfield pairing}
\keyword{$L$--theory}
\subject{primary}{msc2000}{57Q45}
\subject{secondary}{msc2000}{57Q60}
\subject{secondary}{msc2000}{57R65}
\subject{secondary}{msc2000}{57R67}

\arxivreference{}  
\arxivpassword{}   

\DeclareMathOperator{\im}{im}

\DeclareMathOperator{\Hom}{Hom}

\DeclareMathOperator{\Ext}{Ext}
\DeclareMathOperator{\Tor}{Tor}

\DeclareMathOperator{\Ch}{Ch}
\DeclareMathOperator{\ev}{ev}

\DeclareMathOperator{\closure}{cl}

\newcommand{\eps}{\varepsilon}

\def\A{\mathbb{A}}
\def\B{\mathbb{B}}
\def\C{\mathbb{C}}
\def\D{\mathbb{D}}

\def\H{\mathbb{H}}
\def\R{\mathbb{R}}
\def\Z{\mathbb{Z}}

\def\Q{\mathbb{Q}}

\def\NN{{\mathfrak{N}}}

\def\p{{\mathfrak{p}}}

\def\lmat{\left(\begin{smallmatrix}}
\def\rmat{\end{smallmatrix}\right)}

\def\sm{\setminus}

\def\id{\operatorname{id}}

\def\co{\colon\thinspace}

\theoremstyle{plain}
\newtheorem{theorem}{Theorem}[section]
\newtheorem{proposition}[theorem]{Proposition}
\newtheorem{lemma}[theorem]{Lemma}
\newtheorem{corollary}[theorem]{Corollary}

\newtheorem{question}[theorem]{Question}

\theoremstyle{definition}
\newtheorem{definition}[theorem]{Definition}
\newtheorem{example}[theorem]{Example}

\theoremstyle{remark}

\newtheorem*{remark}{Remark}

\newtheorem*{notation}{Notation}

\newcounter{myenum1}
\newenvironment{flushenumerate}{%
  \begin{list}{\arabic{myenum1}.}%
    {\setlength{\leftmargin}{0pt}}%
     \setlength{\labelwidth}{0pt}
     \setlength{\itemindent}{0.5em}
     \setlength{\labelsep}{0.5em}
     \usecounter{myenum1}}%
  {\end{list}} 

\newenvironment{flushenumerate(i)}{
\begin{enumerate}[(i)]
  \setlength{\leftmargin}{0pt}
}{\end{enumerate}}

\begin{document}
\begin{abstract}
We develop a theory of chain complex double-cobordism for chain complexes equipped with Poincar\'{e} duality. The resulting double-cobordism groups are a refinement of the classical torsion algebraic $L$--groups for localisations of a ring with involution. The refinement is analogous to the difference between metabolic and hyperbolic linking forms.

We apply the double $L$--groups in high-dimensional knot theory to define an invariant for doubly slice $n$--knots. We prove that the ``stably doubly slice implies doubly slice'' property holds (algebraically) for Blanchfield forms, Seifert forms and for the Blanchfield complexes of $n$--knots for $n\geq 1$.

\end{abstract}
\maketitle
\begin{abstract}
We develop a theory of chain complex double-cobordism for chain complexes equipped with Poincar\'{e} duality. The resulting double-cobordism groups are a refinement of Ranicki's classical torsion algebraic $L$--groups for localisations of a ring with involution. The refinement is analogous to the difference between metabolic and hyperbolic linking forms.

We apply the double $L$--groups in high-dimensional knot theory to define an invariant for doubly slice $n$--knots. We prove that the `stably doubly slice implies doubly slice' property holds (algebraically) for Blanchfield forms, Seifert forms and for the Blanchfield complexes of $n$--knots for $n\geq 1$.
\end{abstract}

\section{Introduction}\label{sec:intro}

In this paper we develop new algebraic methods in the study of linking forms and in the algebraic cobordism theory of chain complexes equipped with Poincar\'{e} duality. Taking $A$ to be a ring with involution and $S$ a multiplicative subset, we will use our new methods to refine Ranicki's torsion algebraic $L$--groups $L^n(A,S)$. Our refinements are called the \emph{double $L$--groups} $DL^n(A,S)$. Algebraically, our new methods are motivated by Levine's work \cite{MR1004605} on the difference between \emph{metabolic} and \emph{hyperbolic} linking forms. Our main innovation is a generalisation of this algebraic distinction to the setting of chain complexes with Poincar\'{e} duality by means of a notion of algebraic \emph{double-cobordism}.

Our topological motivation, just as Levine's, comes from high-dimensional knot theory. Fox \cite[p.\ 138]{MR0140099} posed the question of which knots $K\co S^n\hookrightarrow S^{n+2}$ are the intersection of an $(n+1)$--unknot and the equator $S^{n+2}\subset S^{n+3}$. Such knots are called \emph{doubly slice}. In the case $n=1$, this question has enjoyed a recent revival of interest in the work of Kim \cite{MR2218756}, Meier \cite{Meier:2014qf} and Livingston-Meier \cite{Livingston:2015kq}. The $n$--dimensional double knot-cobordism group $\mathcal{DC}_n$ is the quotient of the monoid of $n$--knots by the submonoid of doubly slice knots. Using a chain complex knot invariant, we will define a homomorphism from $\mathcal{DC}_n$ to a certain double $L$--group\begin{equation}\label{eq:obstruction}\sigma^{DL}\co \mathcal{DC}_n\to DL^{n+1}(\Lambda,P),\end{equation}where $\Lambda=\Z[z,z^{-1},(1-z)^{-1}]$ and $P$ is the set of Alexander polynomials. In particular our homomorphism uses the entire chain complex of the knot exterior to obstruct the property of being doubly slice.

\subsection{The slice and doubly slice problems}

Detecting doubly slice knots is intimately related to detecting slice knots, as follows. Working in the topologically locally flat category, an (oriented) $n$--knot $K\co S^n\hookrightarrow S^{n+2}$ is called \emph{slice} if it admits a \emph{slice disc}, that is an oriented embedding of pairs $$(D,K)\co(D^{n+1},S^n)\hookrightarrow (D^{n+3},S^{n+2}).$$ The monoid $Knots_n$, of $n$--knots under connected sum, modulo the submonoid of slice $n$--knots is the $n$--dimensional \emph{knot-cobordism group} $\mathcal{C}_n$. So a doubly slice knot is exactly a knot $K$ which admits two \emph{complementary} slice discs $(D_\pm,K)$, that is discs that glue together along $K$ to form the $(n+1)$--unknot.

Most questions that can be asked about slice knots can be asked about doubly slice knots as well, although the answer in the doubly slice case will almost always be more difficult to come by. When $n=1$, there is a still a great deal left to understand about slice knots and the knot-cobordism group (in both the smooth and topological categories). So new results for doubly slice knots here can only go so far without new slice results. In contrast, when $n>1$ Kervaire \cite{MR0189052} and Levine \cite{MR0246314} completely solved the (singly) slice problem in both the smooth and topological categories. They showed that all even-dimensional knots are slice and, using algebraic results of Stoltzfus \cite{MR0467764}, we now know that when $k>0$:\begin{equation}\label{eq:concordance}\mathcal{C}_{2k+1}\cong  \bigoplus_\infty\Z\oplus \bigoplus_\infty(\Z/2\Z)\oplus \bigoplus_\infty(\Z/4\Z).\end{equation}So perhaps there is hope that we can obtain a substantial classification result for high-dimensional \emph{doubly slice} knots. How far does the high-dimensional (singly) slice solution transfer over to the doubly slice question? Certainly not completely. The first stage of the Kervaire--Levine proof requires one to do surgery on a closed knot exterior $X_K=\closure (S^{n+2}\sm K\times D^2)$ to reduce it to a \emph{simple} knot, that is a knot $K'\sim K\in \mathcal{C}_n$, $n>2$, such that $\pi_r(X_{K'})=\pi_r(S^1)$ for $2r<n+2$. Such knots are then entirely classified in $\mathcal{C}_n$ by the Witt class of the Blanchfield form (see Section \ref{sec:DWgroups} for definitions), whence Equation \ref{eq:concordance}. But this is where the doubly slice case differs, as one consequence of high-dimensional Casson--Gordon invariants defined by Ruberman \cite{MR709569, MR933307} is that this surgery process to obtain a simple knot is generally obstructed within $\mathcal{DC}_n$. There is no `double surgery below the middle dimension' and so now Blanchfield forms are certainly insufficient.

This suggests the approach we have taken in this paper -- we work with a different knot invariant called the \emph{Blanchfield complex} (see \ref{subsec:blanchfield}), that encompasses the entire chain complex of the knot exterior and from which the Blanchfield form can be derived. The Blanchfield complex is a symmetric chain complex over $\Z[\Z]$, whose chain homotopy type is a knot invariant and whose class in the codomain of Equation \ref{eq:obstruction} defines the homomorphism $\sigma^{DL}$. We develop an algebraic framework for the study of doubly slice knots via the Blanchfield complex, which encompasses previous systems based on Witt groups. As well as being interesting in their own right, inroads into this high-dimensional doubly slice problem may shed light on the nature of the low-dimensional problem, revealing which features are typical to both and which may be unique to low-dimensions.

\subsection{Chain complex double-cobordism and double $L$--groups}

The new algebra we develop in Section \ref{sec:DLgroups} to analyse the full chain complex of the exterior of a doubly slice knot is based on Ranicki's Algebraic Theory of Surgery \cite{MR560997,MR566491}. This theory is an algebraic analogue to the cobordism of closed, oriented topological manifolds. The objects $(C,\phi)$ of the theory are chain complexes $C$ equipped with some additional structure $\phi$, capturing algebraic symmetries, such as Poincar\'{e} duality. These objects are then considered under a notion of \emph{algebraic} cobordism (see Subsection \ref{subsec:Ltheory} for definitions).

In Subsection \ref{subsec:doub} of this paper, we will define the concept of an algebraic \emph{double-nullcobordism}. An algebraic double-nullcobordism consists of two algebraic nullcobordisms which glue together in \emph{complementary} way, analogous to complementary slice discs for a doubly slice knot (see \ref{def:comp} for precise definition). Double-cobordism groups are then the set of all $(C,\phi)$ modulo the double-nullcobordant $(C,\phi)$. For a ring with involution $A$ and a localisation of this ring $A\hookrightarrow S^{-1}A$, we make make precise the situations where algebraic double-cobordism groups of various types - which we call the \emph{symmetric double $L$--groups} $DL^n(A)$,  \emph{torsion symmetric double $L$--groups} $DL^n(A,S)$, and \emph{ultraquadratic double $L$--groups} $\widehat{DL}_n(A)$ - will be well-defined.

Of course, once you have defined a new group of algebraic invariants, it is important to be able to work with it and to make calculations. In this direction we introduce a new technique called \emph{algebraic surgery above and below the middle dimension} (see \ref{subsec:surgery}), to prove the following skew 2--fold periodicity result in some double $L$--groups:

\medskip

{\bf Theorem (\ref{thm:aboveandbelow} and \ref{cor:aboveandbelow})}\qua {\sl For any ring with involution $A$, which has homological dimension 0, and for $n\geq 0$, there are isomorphisms:}\[\begin{array}{rrcl}
\overline{S}\co &DL^n(A,\eps)&\xrightarrow{\cong} &DL^{n+2}(A,-\eps),\\
\overline{S}\co &\widehat{DL}_n(A,\eps)&\xrightarrow{\cong} &\widehat{DL}_{n+2}(A,-\eps),\end{array}\]{\sl so that for $k\geq 0$} \[\begin{array}{ll}DL^{2k+1}(A,\eps)=0,&DL^{2k}(A,\eps)\cong DL^0(A,(-1)^k\eps),\\ \widehat{DL}_{2k+1}(R,\eps)=0, &\widehat{DL}_{2k}(R,\eps)\cong\widehat{DL}_0(R,(-1)^k\eps).\end{array} \]

\medskip

The question of calculating double $L$--groups under these hypotheses is thus reduced the the problem of calculating the groups in dimension 0, to which we turn in Section \ref{sec:DWgroups}.

In Section \ref{sec:DWgroups} we work with the classical tools of linking forms and Seifert forms. As mentioned, we are able to use these tools to make calculations of double $L$--groups in terms of what we called in a previous paper \emph{Double Witt groups} \cite{OrsonA}. If a form (resp.\ linking form/Seifert form) admits a maximally self-annihilating submodule then it is called \emph{metabolic}. If it admits two such submodules, that are moreover complementary, it is called \emph{hyperbolic}. Witt groups are defined by taking forms modulo stably metabolic forms and double Witt groups are defined by taking forms modulo stably hyperbolic forms. We prove the following:

{\bf Proposition (\ref{prop:isogroups} and \ref{iso})}\qua {\sl For any ring with involution $A$, the 0--dimensional ultraquadratic double $L$--group is isomorphic to the double Witt group of Seifert forms}\[
 \widehat{DL}_0(A,\eps)\cong \widehat{DW}_\eps(A).\]{\sl If there exists a central $s\in A$ such that $s+\overline{s}=1$, the 0--dimensional torsion double $L$--group is isomorphic to the double Witt group of linking forms}\[DL^0(A,S,\eps)\cong DW^\eps(A,S).\]

From this, we are able to compute many double $L$--groups in terms of signature invariants of Seifert forms and of linking forms (see Example \ref{ex:calculation}), using our results in \cite{OrsonA}. In particular, when $A$ is a Dedekind domain, one may apply \cite[Theorem 3.26]{OrsonA} to see that the forgetful functor from the ultraquadratic double $L$--groups of $A$ to the ultraquadratic single $L$--groups of $A$ has (countably) infinitely generated kernel. This gives a first idea of just how big the double $L$-groups are.

\subsection*{The stably doubly slice question}

In our knot theoretical application of double $L$--theory, as well as describing a new algebraic framework for working with doubly slice knots, we prove some new algebraic results related to the stably vs.\ unstably doubly slice question.

By definition, two knots $K,K'$ are equivalent in $\mathcal{DC}_n$ whenever there exist doubly slice knots $J,J'$ such that $K\#J\simeq K'\#J'$. In particular, $K$ vanishes in $\mathcal{DC}_n$ if and only if $K\#J$ is doubly slice for some doubly slice $J$ (we say $K$ is \emph{stably doubly slice}). Arguably the most important question for doubly slice knots is:

\begin{question}\label{q:stablyhypishyp} If an $n$--knot $K$ is stably doubly slice, is it necessarily doubly slice?\end{question} 

An answer to this question would determine whether the double knot-cobordism classes not only obstruct, but moreover characterise doubly slice knots. The algebraic versions of this question are thus interesting for any group-valued doubly slice invariant, such as our $\sigma^{DL}$.

In this spirit we will prove the following set of results. Recall $\Lambda=\Z[z,z^{-1},(1-z)^{-1}]$ and $P$ is the set of Alexander polynomials.

{\bf Theorem (\ref{thm:stablyresults})}\qua {\sl Suppose for $n\geq1$ that an $n$--knot $K$ is stably doubly slice. Then the double $L$--class of the Blanchfield complex $\sigma^{DL}(K)\in DL^{n+1}(\Lambda,P)$ vanishes. If $n=2k+1$ then the Witt classes of the Blanchfield form $\sigma^{DW}(K)\in DW^{(-1)^k}(\Lambda,P)$, and the Witt class of any choice of Seifert form $\sigma_{\widehat{DW}}(K)\in \widehat{DW}_{(-1)^{k+1}}(\Z)$, vanish.}

As a consequence, if there were a stably doubly slice $n$--knot which is not doubly slice, then this would be undetectable by any of the invariants $\sigma^{DL}(K)$, $\sigma^{DW}(K)$, $\sigma_{\widehat{DW}}(K)$. In this paper we work with coefficients $\Z[\Z]$, so we note that Theorem \ref{thm:stablyresults} does not cover, for example, the twisted Blanchfield forms of Cochran-Orr-Teichner \cite{MR1973052}. The possibility of using different fundamental groups is discussed in the closing remarks of the paper.

The result in the case of the Blanchfield form over $\Z[\Z]$ is not new, but is a reproof of a Theorem of Bayer-Fl\"{u}ckiger and Stoltzfus \cite{MR833015} -- we include it because our proof is an application of the techniques of double $L$--theory and as such uses very different methods. Indeed, the `stably hyperbolic implies hyperbolic' results (Corollaries \ref{cor:stablyhypishypseif} and \ref{cor:stablyhypishyp}) which led to Theorem \ref{thm:stablyresults} were a surprising by-product of the development of the double $L$--groups and the low-dimensional double Witt group isomorphisms (Propositions \ref{prop:isogroups} and \ref{iso}) we obtained.

\subsection{Organisation}

In \ref{subsec:conventions} we lay out the algebraic conventions we are using. In \ref{subsec:Ltheory} we will need to recall some elements of Ranicki's Algebraic Theory of Surgery which we require later, and as this is not a common tool, we have tried to give the reader a useful introduction with many references. We pay particular emphasis to the $\eps$--ultraquadratic version of the $L$--theory machinery as there is very little in the literature about this.

In \ref{subsec:doub} we define the ultraquadratic double $L$--groups over a ring with involution $R$, and both the projective and torsion symmetric double $L$--groups over a ring with involution $A$ admitting a central element $s$ such that $s+\overline{s}=1$. In \ref{subsec:surgery} we investigate structure and periodicity results in double $L$--theory via the \emph{skew-suspension} map. We introduce our technique of \emph{algebraic surgery above and below the middle dimension} in order to prove periodicity in certain double $L$--groups.

In Section \ref{sec:DWgroups} we relate the double $L$--groups to the double Witt groups we introduced in \cite{OrsonA}. Firstly, we show how to interpret the 0--dimensional double $L$--groups as double Witt groups. This allows calculation of double $L$--groups for some rings and also establishes `stably hyperbolic implies hyperbolic' results for Seifert forms and linking forms. Secondly, we show (for certain rings) how to directly extract double Witt invariants from a class in an odd-dimensional double $L$--group, which makes the connection between the Blanchfield complex and Blanchfield form later.

In Section \ref{sec:knots} we relate double $L$--theory to the original topological motivation: the doubly slice problem. We recall and elaborate on the construction of Ranicki's \emph{Blanchfield complex} knot invariant. We then prove the claimed doubly slice obstruction of Equation \ref{eq:obstruction} and lay out the consequences of combining this with the algebraic results of Sections \ref{sec:DLgroups} and \ref{sec:DWgroups}.

\subsection{Acknowledgements}

This work follows from the author's PhD thesis at the University of Edinburgh and was supported by the EPSRC. The author would like to thank his advisor Andrew Ranicki for his patient advice and guidance in the preparation of this work. The author would also like to thank the anonymous referee, whose detailed reading and many excellent suggestions have improved this article greatly.

\section{Double $L$--theory}\label{sec:DLgroups}

\subsection{Algebraic conventions and localisation}\label{subsec:conventions}

In the following, $A$ (or sometimes $R$) will be a ring with unit and involution. The involution is denoted\[\overline{\phantom{A}}\co A\to A;\qquad a\mapsto \overline{a}.\]Using the involution we define a way of switching between left and right modules, which will permit an efficient way of describing sesquilinear pairings between left $A$--modules later. A left $A$--module $P$ may be regarded as a right $A$--module $P^t$ by the action \[P^t\times A\to P^t;\qquad (x,a)\mapsto \overline{a}x.\]Similarly, a right $A$--module $P$ may be regarded as a left $A$--module $P^t$. Unless otherwise specified, the term `$A$--module' will refer to a left $A$--module. Given two $A$--modules $P,Q$, the tensor product is an abelian group denoted $P^t\otimes_A Q$. We will sometimes write simply $P\otimes Q$ to ease notation, but the right $A$--module structure $P^t$ is implicit, so that for example $x\otimes ay=\overline{a}x\otimes y$.

In the following, $S\subset A$ will always be a \emph{multiplicative} subset, that is a set with the following properties:

\begin{flushenumerate(i)}
\item $st\in S$ for all $s,t\in S$,
\item $sa=0\in A$ for some $s\in S$ and $a\in A$ only if $a=0\in A$,
\item $\overline{s}\in S$ for all $s\in S$,
\item $1\in S$.
\item For $a\in A, s\in S$ there exists $b,b'\in A$, $t, t'\in S$ such that $at=sb$ and $t'a=b's$.
\end{flushenumerate(i)}The \emph{localisation of $A$ away from $S$} is $S^{-1}A$, the ring with involution formed of equivalence classes of pairs $(a,s)\in A\times S$ under the relation $(a,s)\sim (b,t)$ if and only if there exists $c,d\in A$ such that $ca=db$ and $cs=dt$. We say the pair $(A,S)$ \emph{defines a localisation} and denote the equivalence class of $(a,s)$ by $a/s\in S^{-1}A$. (The use of (v) above, the `two-sided Ore condition', ensures an isomorphism between the left and right localisations $S^{-1}A$ and $AS^{-1}$.) If $P$ is an $A$--module denote $S^{-1}P:=S^{-1}A\otimes_AP$ and write the equivalence class of $(a/s)\otimes x$ as $ax/s$. Similarly, if $f\co P\to Q$ is a morphism of $A$--modules then there is induced a morphism of $S^{-1}A$--modules $S^{-1}f=1\otimes f\co S^{-1}P\to S^{-1}Q$. Generally $i\co P\to S^{-1}P$ is injective if and only if $\Tor^A_1(S^{-1}A/A,P)$ vanishes. This happens, for instance, when $P$ is a projective module. If $S^{-1}P=0$ then the $A$--module $P$ is called \emph{$S$--torsion}, and more generally define the \emph{$S$--torsion of $P$} to be $TP:=\ker(P\to S^{-1}P).$

\subsection*{Torsion modules and duality}

Define a category\[\A(A)=\{\text{finitely generated (\text{f.g.}), projective $A$--modules}\},\] with $A$--module morphisms. An $A$--module $Q$ has \emph{homological dimension $m$} if it admits a resolution of length $m$ by f.g.\ projective $A$--modules, that is there is an exact sequence\[0\to P_m\to P_{m-1}\to\dots\to P_0\to Q\to 0,\]with $P_i$ in $\A(A)$. If this condition is satisfied by all $A$--modules $Q$ we say $A$ is of homological dimension $m$. If $(A,S)$ defines a localisation, define a category \[\H(A,S)=\{ \text{f.g.\ $S$--torsion $A$--modules of homological dimension 1}\}\]with $A$--module morphisms. $\A(A)$ has a good notion of duality, coming from the $\Hom$ functor, and $\H(A,S)$ has a corresponding good notion of `torsion duality' as we now explain.

Given $A$--modules $P$, $Q$, we denote the additive abelian group of $A$--module homomorphisms $f\co P\to Q$ by $\Hom_A(P,Q)$. The \emph{dual} of an $A$--module $P$ is the $A$--module \[P^*:=\Hom_A(P,A)\]where the action of $A$ is $(a,f)\mapsto (x\mapsto f(x)\overline{a})$. If $P$ is in $\A(A)$, then there is a natural isomorphism\[\sm-\co P^t\otimes Q\xrightarrow{\cong}\Hom_A(P^*,Q);\qquad x\otimes y\mapsto (f\mapsto \overline{f(x)}y).\]In particular, using the natural $A$--module isomorphism $P\cong P^t\otimes A$, there is a natural isomorphism \[P\xrightarrow{\cong} P^{**};\qquad x\mapsto (f\mapsto \overline{f(x)}).\]Using this, for any $A$--module $Q$ in $\A(A)$ and $f\in\Hom_A(Q,P^*)$ there is a \emph{dual morphism}\[f^*\co P\to Q^*;\qquad x\mapsto (y\mapsto \overline{f(y)(x)}).\]To proceed similarly in the category $\H(A,S)$, recall the following well-known results in homological algebra:

\begin{lemma}\label{lem:ext1}Suppose $T$ is a f.g.\ $A$--module. \begin{flushenumerate(i)}
\item If $T$ has homological dimension 1 and $T^*=0$ then there is a natural isomorphism of $A$--modules $T\cong \Ext^1_A(\Ext^1_A(T,A),A)$.
\item If $(A,S)$ defines a localisation and $T$ is $S$--torsion, then $T^*=0$ and there is a natural isomorphism\[\Ext^1_A(T,A) \cong\Hom_A(T,S^{-1}A/A).\]
\end{flushenumerate(i)}
\end{lemma}

Lemma \ref{lem:ext1} justifies the following definitions. The \emph{torsion dual} of a module $T$ in $\H(A,S)$ is the module\[T^\wedge:=\Hom_A(T,S^{-1}A/A)\]in $\H(A,S)$ with the action of $A$ given by $(a,f)\mapsto (x\mapsto f(x)\overline{a})$. There is a natural isomorphism\[T\xrightarrow{\cong} T^{\wedge\wedge};\qquad x\mapsto (f\mapsto \overline{f(x)}),\]and for $R,T$ in $\H(A,S)$, $f\in\Hom_A(R,T^\wedge)$ there is a \emph{torsion dual morphism}\[f^\wedge\co T\to R^\wedge;\qquad x\mapsto(y\mapsto\overline{f(y)(x)}).\]

\subsection*{Chain complex conventions}

Given chain complexes $(C,d_C),(D,d_D)$ of $A$--modules a \emph{chain map of degree $n$} is a collection of morphisms $f_r\co C_r\to D_{r+n}$ with $d_Df_r=(-1)^nf_{r-1}d_C$. The category of chain complexes of $A$--modules with morphisms degree 0 chain maps is denoted $\Ch(A)$. A chain complex $C$ in $\Ch(A)$ is \emph{finite} if it is concentrated in finitely many dimensions, and \emph{positive} if $H_r(C)=0$ for $r<0$. The category of finite, positive chain complexes of objects of $\A(A)$ is denoted $\B_+(A)$. If $C$ is in $\Ch(A)$, let $C^t$ denote the chain complex of f.g.\ projective, right $A$--modules $(C^t)_r:=(C_r)^t$. The \emph{dual chain complex} of $C$ in $\Ch(A)$ is $C^{-*}$ in $\Ch(A)$ with modules $(C^{-*})_r:=(C_{-r})^*=\co C^{-r}$ and differential $(-1)^rd^*_C\co C^{-r}\to C^{-r+1}$. The \emph{suspension} of $C$ in $\Ch(A)$ is the chain complex $\Sigma C$ in $\Ch(A)$ with modules $(\Sigma C)_r=C_{r-1}$ and differential $d_{\Sigma C}=d_C$. The \emph{desuspension} $\Sigma^{-1}C$ is defined by $\Sigma(\Sigma^{-1}C)=C$. Morphisms $f,f'\co C\to D$ are \emph{chain homotopy equivalent} if there exists a collection of $A$--module morphisms $h=\{h_r\co C_r\to D_{r+1}\,|\,r\in \Z\}$ so that $f-f'=d_Dh+hd_C$, in which case the collection is called a \emph{chain homotopy} and we write $h\co f\simeq f'$. A morphism $f\co C\to D$ is a \emph{chain homotopy equivalence} if there exists a morphism $g\co D\to C$ such that $fg\simeq 1_D$ and $gf\simeq 1_C$. The homotopy category of $\B_+(A)$ is denoted $h\B_+(A)$.

For $C,D$ in $\Ch(A)$, there are chain complexes of $\Z$--modules \begin{gather*}(C^t\otimes_A D)_r:=\bigoplus_{p+q=r}C^t_p\otimes_A D_q;\quad d(x\otimes y)=x\otimes d_D(y)+(-1)^qd_C(x)\otimes y,\\(\Hom_A(C,D))_r:=\prod_{q-p=r}\Hom_A(C_p,D_q);\quad d(f)=d_D(f)-(-1)^rfd_C,\end{gather*}and the \emph{slant map} is defined as \[\setminus-\co C^t\otimes_A D\to \Hom_A(C^{-*},D);\quad x\otimes y\mapsto (f\mapsto \overline{f(x)}y).\]In the sequel we will often write $C\otimes D$ in place of $C^t\otimes_AD$ in order to ease notation. If $C,D$ are (chain homotopy equivalent to) objects of $\B_+(A)$ then the slant map is a chain (homotopy) equivalence. When $C,D$ are chain homotopy equivalent to objects of $\B_+(A)$, there is an isomorphism of groups\[\{\text{$n$--cycles in $\Hom_A(C,D)$}\}\cong\{\text{chain maps of degree $n$ from $C$ to $D$}\}.\]Combining the above, when $C, D$ are chain homotopy equivalent to objects of $\B_+(A)$ and $\psi\in (C\otimes D)_n$, we will write the associated morphisms\[\psi_0\co C^{n-r}\to D_r,\qquad r\in\Z.\]When $\psi$ is moreover a cycle, $\psi_0$ describes a chain map.

A morphism $f\co C\to D$ in $\Ch(A)$ is a \emph{cofibration} if it is degreewise split injective and a \emph{fibration} if it degreewise split surjective. A sequence of morphisms in $\Ch(A)$ is a \emph{(co)fibration sequence} if each morphism in the sequence is a (co)fibration. The \emph{algebraic mapping cone} of $f$ is the chain complex $C(f)$ in $\Ch(A)$ with $C(f)_r=D_r\oplus C_{r-1}$ and \[d_{C(f)}=\left(\begin{matrix} d_D&(-1)^{r-1}f\\0&d_C\end{matrix}\right)\co D_r\oplus C_{r-1}\to D_{r-1}\oplus C_{r-2}.\]There is an obvious inclusion morphism $e\co D\to C(f)$ and the composite $ef\co C\to C(f)$ is easily seen to be nullhomotopic (see \cite[\textsection 11]{tibor} for more details of mapping cones). A \emph{homotopy cofibration sequence} is a sequence of morphisms in $\Ch(A)$ such that any two successive morphisms \[\xymatrix{C\ar[r]^-{f}&D\ar[r]^g&E}\] have nullhomotopic composition and such that any choice of nullhomotopy $j\co gf\simeq 0$ induces a chain equivalence $\Phi_j\co C(f)\simeq E$. A sequence of morphisms in $\Ch(A)$ is a \emph{homotopy fibration sequence} if the dual sequence of morphisms is a homotopy cofibration sequence. Using the obvious projection morphisms $\text{proj}\co C(f)\to \Sigma C$, every morphism $f\co C\to D$ in $\Ch(A)$ has an associated \emph{Puppe sequence}\[\dots\to\Sigma^{-1}D\to \Sigma^{-1}C(f)\xrightarrow{\Sigma^{-1}\text{proj}} C\xrightarrow{f} D\xrightarrow{e} C(f)\xrightarrow{\text{proj}} \Sigma C\xrightarrow{\Sigma f} \Sigma D\to \dots\]which is both a homotopy fibration sequence and a homotopy cofibration sequence. In particular this shows that in $\Ch(A)$, homotopy fibration sequences agree with homotopy cofibration sequences. Given diagrams\[\xymatrix{D&C\ar[l]_-{f}\ar[r]^-{f'}& D'}\quad\text{ and }\quad\xymatrix{D\ar[r]^-{g}&E&D'\ar[l]_-{g'}}\]the \emph{homotopy pushout} and \emph{homotopy pullback} are given respectively by \[D\cup_C D':=C\left(\lmat -f\\f'\rmat \co C\to D\oplus D'\right)\]\[D\times_E D':= \Sigma^{-1}C((g\,\,-g')\co D\oplus D'\to E).\]A \emph{homotopy commuting square} $\Gamma$ in $\Ch(A)$ is a diagram \[\xymatrix{C\ar@{~>}[dr]^-{h}\ar[r]^-{f'}\ar[d]_-{f}&D'\ar[d]^-{g'}\\D\ar[r]^{g}&E}\]consisting of a square of morphisms $f,f',g,g'$ in $\Ch(A)$ together with a homotopy $h\co g'f'\simeq gf$. A homotopy commuting square induces the obvious maps of cones\[C(g',f)\co C(f')\to C(g),\qquad C(g,f')\co C(f)\to C(g').\]Taking cones again, there is not just homotopy equivalence, but actual equality $C(C(g',f))=C(C(g,f'))$. We define the \emph{iterated cone on $\Gamma$} to be that chain complex \[C(\Gamma)=C(C(g',f))=C(C(g,f')).\]Note that, strictly speaking, the morphisms $C(g',f)$ and $C(g,f')$ in $\Ch(A)$, and hence the complex $C(\Gamma)$, depend on the choice of $h$ but this is suppressed from the notation. A \emph{homotopy pushout square} is a homotopy commuting square $\Gamma$ such that the induced map $\Phi_h\co D\cup_C D'\to E$ is a homotopy equivalence. A \emph{homotopy pullback square} is defined analogously using the homotopy pullback.

\subsection{Structured chain complexes and algebraic cobordism}\label{subsec:Ltheory}

In this section we will recall for the reader's convenience some algebraic definitions and constructions from Ranicki's Algebraic Theory of Surgery \cite{MR560997,MR566491, MR620795}, a theory whose development was originally motivated by the challenge to provide a `chain complex cobordism' reformulation for the quadratic surgery obstruction groups of Wall \cite{MR0431216}. This is very far from a complete account and the reader will sometimes be given (detailed) references to the literature for the basics of the theory. In this paper will be primarily working with a version of the algebraic machinery called \emph{symmetric $L$--theory} and there is enough in the literature about this for us to rely on references fairly frequently.

In order to prove results about Seifert forms later, we will also need to work with a version called \emph{ultraquadratic $L$--theory} (originally defined in Ranicki \cite[p.\ 814]{MR620795}). There is very little written about this type of $L$--theory so we will give more careful proofs and constructions here. Ultraquadratic $L$--theory is algebraically simpler than the general symmetric $L$--theory, but is less robust as an algebraic tool. Notably, in Proposition \ref{prop:surgery}, we will prove the existence of obstructions to \emph{algebraic surgery} in this setting (where in the symmetric case algebraic surgery is always unobstructed). In the ultraquadratic theory the objects are pairs $(C,\psi)$, given by a chain complex $C$ in $h\B_+(A)$ equipped with an $n$--cycle $\psi\in (C\otimes C)_n$, or equivalently a chain map $\psi_0\co C^{n-*}\to C_*$. The objects $(C,\psi)$ of ultraquadratic $L$--theory arise from the geometric situation of a degree 1 map of closed oriented topological manifolds \[f\co M\to X\times S^1,\]that is covered by a stable map of topological normal bundles and such that $f$ is a $\Z[\pi_1(X)]$--homology equivalence (see Ranicki \cite[p.\ 818]{MR620795} for precise details on the construction of $(C,\psi)$ from this geometric setup). In particular, when $X$ is a disc $D^{n+1}$, the exterior $M^{n+2}$ of an $n$--knot $S^n\hookrightarrow S^{n+2}$ possesses such a map (rel.\ boundary) as we explain in Section \ref{sec:knots}.

\begin{definition}A \emph{half-unit} $s\in A$ is a central element such that $s+\overline{s}=1\in A$.
\end{definition}

When $A$ contains a half-unit (for instance, when 2 is invertible in $A$), the structured chain complexes of general symmetric $L$--theory always simplify to those of ultraquadratic $L$--theory. This appears to be well-known to experts but not written down, so we will make the proofs of this clear. However, it is not the case that the resultant $L$--theories are the same as a symmetric algebraic cobordism does not necessarily improve to an ultraquadratic one.

\subsection*{Chain complexes with symmetric structure}

From now on, take $\eps\in A$ to be a central unit such that $\eps\overline{\eps}=1$ ($\eps=\pm1$ will be a common choice). The cyclic group of order 2 is denoted $\Z/2\Z=\{1,T\}$. Let $C$ be in $\Ch(A)$ and define the standard $\eps$--involution\begin{eqnarray*}T=T_\eps\co C_p^t\otimes_A C_q&\to& C_q^t\otimes_A C_p\\x\otimes y&\to& \eps(-1)^{pq}y\otimes x\end{eqnarray*}so that $C^t\otimes_A C$ in $\Ch(\Z)$ may be regarded as a chain complex of $\Z[\Z/2\Z]$ modules. The standard free $\Z[\Z/2\Z]$--module resolution $W$ of $\Z$ is the chain complex $W_i=\Z[\Z/2\Z]$, $i\geq 0$\[W:\qquad\dots\to W_3\xrightarrow{1-T} W_2 \xrightarrow{1+T} W_1 \xrightarrow{1-T} W_0 \to 0.\]The `homotopy fixed points' of the involution $T_\eps$ on $C^t\otimes_A C$ are in the form of the complex of $\Z$--modules\[W^\%C=W_\eps^\%C:=\Hom_{\Z[\Z/2\Z]}(W,C^t\otimes_A C).\] Given a morphism $f\co C\to D$, the $W^\%$ construction induces a morphism of abelian groups\[f^\%\co W^\%C\to W^\% D\]so that $W^\%$ is a functor $\Ch(A)\to \Ch(\Z)$. It is possible to show (see Ranicki \cite[p.\ 101]{MR560997}) given a homotopy $h: f_1 \simeq f_2\co C\to D$ there exists a (non-canonical) choice of homotopy $h^\%: f_0^\%\simeq f_1^\%\co W^\% C\to W^\% D$. Unless clarification is needed, we suppress the `$\eps$' from the notation of $T$ and $W^\%$.

\begin{definition}For $n\geq 0$, an \emph{$n$--dimensional $\eps$--symmetric structure on $C$ in $h\B_+(A)$} is an $n$--dimensional cycle $\phi\in W^\%C_n$. The pair $(C,\phi)$ is called an \emph{$n$--dimensional $\eps$--symmetric complex}. See \cite[\textsection 20.4]{MR1713074} for definitions of morphisms and homotopy equivalences of symmetric complexes.


For $n\geq 0$, an \emph{$(n+1)$--dimensional $\eps$--symmetric structure on a morphism $f\co C\to D$ in $h\B_+(A)$} is an $(n+1)$--dimensional cycle $(\delta\phi,\phi)\in C(f^\%)_{n+1}\cong W^\%D_{n+1}\oplus W^\%C_{n}$ (this notation indicates $\phi$ is a cycle and $\delta\phi$ is a nullhomotopy of $f^\%(\phi)$). The pair $(f\co C\to D,(\delta\phi,\phi))$ is called an \emph{$(n+1)$--dimensional $\eps$--symmetric pair}. 

\end{definition}

\begin{remark}We have chosen to describe symmetric chain complexes in terms of cycles as in \cite{MR1713074}, rather than homology classes and $Q$--groups as in \cite{MR560997} or \cite{MR620795}. For more detailed information on the perspective we are using, see \cite[\textsection 20]{MR1713074}. 
\end{remark}

In order to define algebraic cobordism groups we will require an algebraic analogue of the glueing of manifolds along a common boundary component. Given $(n+1)$--dimensional $\eps$--symmetric pairs \[x:=((f\,\,f')\co C\oplus C'\to D,(\delta\phi,\phi\oplus \phi')),\qquad x':=((\tilde{f}'\,\,f'')\co C'\oplus C''\to D',(\delta'\phi,\phi'\oplus \phi'')),\] there is an $(n+1)$--dimensional $\eps$--symmetric pair\[x\cup x':=((g\,\,g'')\co C\oplus C''\to D\cup_{C'}D',(\delta\phi\cup_{\phi'}\delta'\phi,\phi\oplus \phi''))\] called \emph{algebraic union of $x$ and $x'$ along $(C',\phi')$}. We refer the reader to Ranicki \cite[\textsection 1.7]{MR620795} or Crowley-L\"{u}ck-Macko \cite[\textsection 11.4.2]{tibor} for details of algebraic glueing.

In order to work relative to the boundary of a knot exterior in Section \ref{sec:knots} we will need a notion of cobordism of pairs in the algebraic setting, which is given by \emph{triads}:

\begin{definition}For $n\geq 0$, an \emph{$(n+2)$--dimensional $\eps$--symmetric structure on a homotopy commuting square $\Gamma$}\[\xymatrix{C\ar@{~>}[dr]^-{h}\ar[r]^-{f'}\ar[d]_-{f}&D'\ar[d]^-{g'}\\D\ar[r]^{g}&E}\]\emph{in $h\B_+(A)$} is a quadruple $(\Phi,\delta\phi,\delta'\phi,\phi)$ such that there are defined $(n+1)$--dimensional $\eps$--symmetric pairs \[(f\co C\to D,(\delta\phi,\phi)),\quad(f'\co C\to D',(\delta'\phi,\phi))\] and an $(n+2)$--dimensional $\eps$--symmetric pair \[(g''\co D\cup_CD'\to E,(\Phi,\delta\phi\cup_\phi\delta'\phi))\]with $g''$ the map induced by universality of the pushout $D\cup_C D'$. The pair $(\Gamma, (\Phi,\delta\phi,\delta'\phi,\phi))$ is called a \emph{$(n+2)$--dimensional $\eps$--symmetric triad}.

\end{definition}

The reader is referred to Ranicki \cite[\textsection 1.3, \textsection 2.1]{MR620795} for a more complete discussion of the algebraic theory of triads.

\subsection*{Chain complexes with ultraquadratic structure}

In order to study the chain complex version of Seifert forms for a knot we will need to look at the version of the $L$--theory machinery called \emph{ultraquadratic $L$--theory}. To build this version, just replace $W$ with the truncated complex $0\to W_0\to 0$ in the construction of $W^\%$. This results similarly in a homotopy functor $\Ch(R)\to \Ch(\Z)$, now simply sending $C\mapsto C\otimes C$ and \[(f\co C\to D)\mapsto (f\otimes f\co C\otimes C\to D\otimes D).\]Recall that for $C$ in $h\B_+(R)$, the slant map is a chain homotopy equivalence $C\otimes C\simeq \Hom_R(C^{-*},C)$ and sends a cycle $\psi\in (C\otimes C)_n$ to a chain map $\psi_0\co C^{n-*}\to C$, so in this `truncated' version all structure is governed by this single chain map.

\begin{definition}For $n\geq 0$, an \emph{$n$--dimensional $\eps$--ultraquadratic structure on $C$ in $h\B_+(R)$} is an $n$--dimensional cycle $\psi\in (C\otimes C)_n$. The pair $(C, \psi)$ is called an \emph{$n$--dimensional $\eps$--ultraquadratic complex}. Two $n$--dimensional $\eps$--ultraquadratic complexes $(C,\psi)$, $(C',\psi')$ are homotopy equivalent if there exists a chain homotopy equivalence $h\co C\xrightarrow{\simeq} C'$ such that $(h\otimes h)(\psi)-\psi'$ is a boundary in $C'\otimes C'$.

The definitions of \emph{$\eps$--ultraquadratic pairs} and \emph{$\eps$--ultraquadratic triads} are made analogously to the symmetric case. (In order to define triads, we note that Ranicki's definition of \emph{algebraic glueing} is still valid in the ultraquadratic setting.)
\end{definition}

Here is how to pass from an ultraquadratic structure to a symmetric structure:

\begin{definition}\label{def:symmetrisation}The \emph{symmetrisation} is a map of $\Z[\Z/2\Z]$--module chain complexes\[1+T_\eps\co C\otimes C\to C\otimes C;\qquad \psi\mapsto(1+T_\eps)\psi=\psi+T_\eps\psi=\co \phi.\]The \emph{symmetrisation} of an $n$--dimensional $\eps$--ultraquadratic complex $(C,\psi)$ is the $n$--dimensional $\eps$--symmetric complex $(C,\phi)$ (where we have used the inclusion $W_0\hookrightarrow W$ to identify $\phi$ as a symmetric structure). We may similarly symmetrise pairs and triads.
\end{definition}

What about the passage from a symmetric structure to an ultraquadratic structure?

\begin{proposition}\label{prop:multiple}When $A$ contains a half-unit $s$, the sets of homotopy equivalence classes of the following objects are in natural 1:1 correspondence with one another:
\begin{enumerate}[(i)]
\item $n$--dimensional $\eps$--symmetric complexes $(C,\phi\in (W^\%C)_n)$ over $h\B_+(A)$,
\item pairs $(C,\phi_0\co C^{n-*}\to C)$ where $\phi_0$ is a chain map in $h\B_+(A)$ such that $\phi_0-(T_\eps\phi)_0\in\Hom_A(C^{-*},C)$ is a boundary,
\item $n$--dimensional $\eps$--ultraquadratic complexes $(C,\psi\in (C\otimes C)_n)$.
\end{enumerate}
\end{proposition}

Proposition \ref{prop:multiple} seems to be well-known to experts, but there does not appear to be a proof in the literature.

\begin{proof}When there is a half-unit, the symmetrisation map in \emph{quadratic $L$--theory} (defined in Ranicki \cite{MR560997}) \[1+T_\eps\co W_\%C\to W^\%C\]is a chain homotopy equivalence by the proof of \cite[Proposition 3.3]{MR560997}. Moreover, if an element $\phi\in (W^\%C)_n$ is in the image of the symmetrisation then (by definition) it is entirely described by an element $\phi\in (C\otimes C)_n$. But this element is also in the image of the natural projection $W^\%C\to C\otimes C$. So the projection is a chain homotopy equivalence and the equivalence of (i) and (ii) is proved.

Now, given a pair $(C,\phi_0)$ as in (ii), there is a corresponding pair $(C,\phi\in (C\otimes C)_n)$ and we may define an $n$--dimensional $\eps$--ultraquadratic complex $(C,s\phi)$, which has symmetrisation $(C,\phi)$. Given $n$--dimensional $\eps$--ultraquadratic structures $\psi, \psi'\in (C\otimes C)_n$ with $(1+T_\eps)(\psi-\psi')\simeq 0$, the following commuting square shows that moreover $\psi\simeq \psi'$:\[\xymatrix{W_\%C\ar[r]^-{1+T_\eps}_-{\simeq}&W^\%C\ar[d]_-{\simeq}^-{\text{project}}\\
C\otimes C\ar[u]^-{\text{inclusion}}\ar[r]^-{1+T_\eps}&C\otimes C}\]
\end{proof}

\begin{remark}
Proposition \ref{prop:multiple} does \emph{not} hold analogously for pairs or triads. As an example of this, and using the language of Section \ref{sec:DWgroups}, note that $\eps$--symmetric Seifert forms determine 0--dimensional $\eps$--ultraquadratic complexes and $\eps$--symmetric forms equipped with a lagrangian determine 1--dimensional $\eps$--symmetric pairs (cf. proof of Proposition \ref{prop:metcxseif}, below). But consider that, for example, the symmetrisation of a rational Seifert form for any knot $S^1\hookrightarrow S^3$ is the standard hyperbolic matrix, but not every Seifert form for such a knot admits a metaboliser (there are knots which are not `algebraically slice'). So we see the corresponding symmetric pair has no corresponding ultraquadratic pair.\end{remark}

\subsection*{Algebraic Thom construction}

We now briefly describe Ranicki's algebraic Thom construction \cite[p.\ 46]{MR620795}, which will be required in Section \ref{sec:knots} to change perspective between complexes/pairs and pairs/triads.

Given an $(n+1)$--dimensional $\eps$--symmetric pair $(f\co C\to D,(\delta\phi,\phi))$, recall that a choice of nullhomotopy \[\xymatrix{C\ar@/_1pc/[rr]_{j\co ef\simeq 0}\ar[r]^f&D\ar[r]^e&C(f)}\] induces a morphism $\Phi_{j^\%}\co C(f^\%)\to W^\%C(f)$. Define the $(n+1)$--dimensional cycle $\delta\phi/\phi:=\Phi_{j^\%}(\delta\phi,\phi)\in W^\%C(f)_{n+1}$. The \emph{algebraic Thom construction} for $(f\co C\to D, (\delta\phi,\phi))$ is the $(n+1)$--dimensional $\eps$--symmetric complex $(C(f),\delta\phi/\phi)$.

Given an $(n+1)$--dimensional $\eps$--ultraquadratic pair $(f\co C\to D,(\delta\psi,\psi))$ there is a choice of morphism $\Phi_j\co C(f\otimes f)\to C(f)\otimes C(f)$. Define the $(n+1)$--dimensional cycle $\delta\psi/\psi:=\Phi_j(\delta\psi,\psi)\in C(f)\otimes C(f)$. The \emph{algebraic Thom construction} for $(f\co C\to D,(\delta\psi,\psi))$ is the $(n+1)$--dimensional $\eps$--ultraquadratic complex $(C(f),\delta\psi/\psi)$.

There is a relative version of the algebraic Thom complex. Given an $(n+2)$--dimensional $\eps$--symmetric triad $(\Gamma,(\Phi,\delta\phi,\delta'\phi,\phi))$ consider the induced maps of cones\[\begin{array}{rrcl}\nu=C(g,f')\co &C(f)&\to&C(g')\\\nu'=C(g',f)\co &C(f')&\to &C(g)\end{array}\]By the universal property of the algebraic mapping cone, we obtain morphisms\[\xymatrix{C(\nu^\%)&C(C(g^\%,(f')^\%))\simeq C(C((g')^\%,f^\%))\ar[r]\ar[l]&C((\nu')^\%)}.\]Using the two images of the $(n+2)$--cycle $(\Phi,\delta\phi,\delta'\phi,\phi)$ under these respective morphisms, a triad defines two $(n+2)$--dimensional $\eps$--symmetric pairs \[\begin{array}{rcl}x&=&(\nu\co C(f)\to C(g'),(\Phi/\delta\phi,\delta\phi'/\phi))\\x'&=&(\nu'\co C(f')\to C(g),(\Phi/\delta\phi,\delta\phi'/\phi))\end{array}\]The \emph{relative algebraic Thom construction} for $(\Gamma,(\Phi,\delta'\phi,\delta\phi,\phi))$ is defined to be the set $\{x,x'\}$.

\subsection*{Poincar\'{e} complexes and $L$--groups}

There are chain complex analogues of Poincar\'{e} duality and Poincar\'{e}-Lefschetz duality for $\eps$--symmetric and $\eps$--ultraquadratic structures. We now recall these and hence Ranicki's definition of algebraic cobordism and the algebraic $L$--groups.

The inclusion of chain complexes $W_0\hookrightarrow W$ induces a natural transformation of functors between $\Hom_{\Z[\Z/2\Z]}(W,-)$ and $\Hom_{\Z[\Z/2\Z]}(W_0,-)$. For a given complex $C$ in $\Ch(R)$ this induces \emph{evaluation} morphisms \[\ev\co W^\%C\to C\otimes C;\quad \phi\mapsto \ev(\phi)\]and we write the image of the evaluation $\ev(\phi)\in (C\otimes C)_n$ under the slant map\[\setminus-\co C\otimes C\to \Hom_A(C^{-*},C);\qquad \phi_0:=\sm(\ev(\phi))\co C^{n-*}\to C.\]If $f\co C\to D$ is a morphism in $\Ch(A)$ then there is a relative evaluation morphism\[\ev\co C(f^\%)\simeq\Hom_{\Z[\Z/2\Z]}(W,C(f\otimes f))\to C(f\otimes f);\quad (\delta\phi,\phi)\mapsto \ev(\delta\phi,\phi)\]and the image of the evaluation $\ev(\delta\phi,\phi)\in C(f\otimes f)_{n+1}=(D\otimes D)_{n+1}\oplus (C\otimes C)_n$ under the slant map is written\[\left(\begin{matrix}\delta\phi_0 &0\\0&\phi_0\end{matrix}\right):=\setminus(\ev(\delta\phi,\phi)),\quad \delta\phi_0\co D^{n+1-*}\to D,\quad \phi_0\co C^{n-*}\to C.\]From this there are derived the maps that will play the part of Poincar\'{e}--Lefschetz duality on the chain level\[\begin{array}{rccl}
(\delta\phi_0\,\,f\phi_0)\co &C(f)^{n+1-*}&\to& D,\\
\lmat \delta\phi_0\\\phi_0 f^* \rmat\co &D^{n+1-*}&\to& C(f).\end{array}\]We refer the reader Crowley-L\"{u}ck-Macko \cite[\textsection 11.4.1]{tibor} for full derivation of this.

\begin{definition}An $n$--dimensional $\eps$--symmetric complex $(C,\phi)$ is \emph{Poincar\'{e}} if \[\phi_0\co C^{n-*}\xrightarrow{\simeq} C\] is a chain homotopy equivalence. An $(n+1)$--dimensional $\eps$--symmetric pair $(f\co C\to D,(\delta\phi,\phi))$ is \emph{Poincar\'{e}} if\[(\delta\phi_0\,\,\,f\phi_0)\co C(f)^{n+1-r}\xrightarrow{\simeq} D_r,\qquad r\in\Z\]is a chain homotopy equivalence. Equivalently, the pair is Poincar\'{e} if there is a chain homotopy equivalence \[\left(\begin{matrix}\delta\phi_0\\ (-1)^{n+1-r}\phi_0 f^*\end{matrix}\right)\co D^{n+1-r}\xrightarrow{\simeq} C(f)_r,\qquad r\in\Z.\]An $(n+2)$--dimensional $\eps$--symmetric triad $(\Gamma, (\Phi,\delta\phi,\delta'\phi,\phi))$ is \emph{Poincar\'{e}} if each of the associated pairs is Poincar\'{e}:\[(f\co C\to D,(\delta\phi,\phi)),\quad(f'\co C\to D',(\delta'\phi,\phi)),\quad(g''\co D\cup_CD'\to E,(\Phi,\delta\phi\cup_\phi\delta'\phi)).\]

An $\eps$--ultraquadratic complex/pair/triad is \emph{Poincar\'{e}} if the symmetrisation (Definition \ref{def:symmetrisation}) is a Poincar\'{e} $\eps$--symmetric complex/pair/triad.
\end{definition}

\begin{definition}Two $n$--dimensional $\eps$--symmetric complexes $(C,\phi)$ and $(C',\phi')$ are \emph{cobordant} if there exists an $(n+1)$--dimensional $\eps$--symmetric Poincar\'{e} pair \[(f\co C\oplus C'\to D,(\delta\phi,\phi\oplus -\phi')).\]

Cobordism is an equivalence relation on the set of $\eps$--symmetric Poincar\'{e} complexes such that homotopy equivalent complexes are cobordant (Lemma \ref{lem:welldef}, below). Moreover, the resultant set of cobordism classes forms a group called the \emph{$n$--dimensional $\eps$--symmetric $L$--group of $A$} (see Ranicki \cite[\textsection 3.2]{MR560997} for the checks that this is a group):\[L^n(A,\eps):=\left\{\begin{array}{c}\text{cobordism classes of $n$--dimensional}\\ \text{$\eps$--symmetric Poincar\'{e} complexes}\end{array}\right\},\]with addition and inverses given by:\[(C,\phi)+(C',\phi')=(C\oplus C',\phi\oplus \phi'),\qquad -(C,\phi)=(C,-\phi)\in L^n(A,\eps).\]

\end{definition}

After replacing the word `symmetric' with the word `ultraquadratic', the previous definition transfers verbatim to give:

\begin{definition}The \emph{$n$--dimensional $\eps$--ultraquadratic $L$--group of $A$} is\[\widehat{L}_n(A,\eps):=\left\{\begin{array}{c}\text{cobordism classes of $n$--dimensional}\\ \text{$\eps$--ultraquadratic Poincar\'{e} complexes}\end{array}\right\}.\]
\end{definition}

For the knot theory in Section \ref{sec:knots}, we will be interested in the $L$--theory and double $L$--theory in the category $\H(A,S)$ of f.g.\ $S$--torsion $A$--modules admitting a projective resolution of length 1. For this reason we introduce the category $\C_+(A,S)\subset \B_+(A)$ of chain complexes $C$ that are \emph{$S$--acyclic}. In other words $S^{-1}H_r(C)=0$ for all $r\in\Z$.

Working now with the subcategory $h\C_+(A,S)\subset h\B_+(A)$, we obtain the following restricted notion of algebraic cobordism in the symmetric setting.

\begin{definition}Two $n$--dimensional $S$--acyclic $\eps$--symmetric complexes $(C,\phi)$ and $(C',\phi')$ are \emph{$(A,S)$--cobordant} if there exists an $(n+1)$--dimensional $\eps$--symmetric Poincar\'{e} pair \[(f\co C\oplus C'\to D,(\delta\phi,\phi\oplus -\phi'))\]such that $f$ is a morphism in $h\C_+(A,S)$.

The set of (A,S)-cobordism classes in $h\C_+(A,S)$ forms a group, called the \emph{$n$--dimensional $\eps$--symmetric $L$--group of $(A,S)$}:\[L^n(A,S,\eps):=\left\{\begin{array}{c}\text{$(A,S)$--cobordism classes of $(n+1)$--dimensional}\\ \text{$S$--acyclic $(-\eps)$--symmetric Poincar\'{e} complexes}\end{array}\right\}.\]
\end{definition}

\begin{remark}The choice of the convention `$n+1$' and `$-\eps$' in the definition of $L^n(A,S,\eps)$ follows Ranicki \cite[\textsection 3.2.2]{MR620795}, where this unusual looking choice is explained.
\end{remark}

\subsection{Chain complex double-cobordism and $DL$--groups}\label{subsec:doub}

In this subsection, we develop our theory of chain complex double-cobordism, which results in a refinement of the classical torsion $L$--groups and ultraquadratic $L$--groups described in the previous subsection. We will define 3 types of double $L$--group, which each refine the respective types described in Subsection \ref{subsec:Ltheory}. First, we will define the $\eps$--ultraquadratic double $L$--group of $R$\[\widehat{DL}_n(R,\eps),\]which is well-defined for any coefficient ring with involution. 

\begin{notation}For the rest of the paper, `$A$' will indicate the assumption of a half-unit $s\in A$.\end{notation}

Working over $A$, we will define the projective and torsion $\eps$--symmetric double $L$--groups, \[DL^n(A,\eps)\qquad\text{and}\qquad DL^n(A,S,\eps),\]which are only well-defined when there is a half-unit in the coefficient ring.

\subsection*{Ultraquadratic $DL$--groups}

For the ultraquadratic case we work with the coefficients $R$.

\begin{definition}\label{def:comp}For $n\geq 0$, two cobordisms between $n$--dimensional $\eps$--ultraquadratic Poincar\'{e} complexes $(C,\psi)$ and $(C',\psi')$ \[x_\pm:=(f_\pm\co C\oplus C'\to D_\pm,(\delta_\pm\psi,\psi\oplus-\psi')\in C(f_\pm\otimes f_\pm)_{n+1})\](labelled `$+$' and `$-$') are \emph{complementary} if the chain map\[\left(\begin{smallmatrix}f_+\\f_-\end{smallmatrix}\right)\co C\oplus C'\to D_+\oplus D_-\]is a homotopy equivalence. In which case we say $(C,\psi)$ and $(C',\psi')$ are \emph{double-cobordant} and that the set $\{x_+,x_-\}$ is a \emph{double-cobordism} between them.
\end{definition}

\begin{lemma}\label{lem:welldef}For $n\geq0$, if $(C,\psi)$, $(C',\psi')$ are homotopy equivalent $n$--dimensional, $\eps$--ultraquadratic Poincar\'{e} complexes over $R$, then they are double-cobordant.
\end{lemma}

\begin{proof} Let $h\co (C,\psi)\to(C',\psi')$ be a given homotopy equivalence with homotopy inverse $g$. Write the chain map $\phi=\psi+T_\eps\psi$, with choice of chain homotopy inverse $\phi^{-1}$. Define a chain map $e=\psi\phi^{-1}$, so that $e^*\simeq \phi^{-1}T_\eps\psi$. Consider the chain map\[(h(1-e)\,\, -heg)\co C\oplus C'\to C'\]and calculate \[\begin{array}{rcl}

h(1-e)\psi(1-e)^*h^*-(heg)\psi'(heg)^*&\simeq & h(\psi-(e\psi+\psi e^*))h^*\\
&\simeq& h(\psi-(e\psi+ eT_\eps\psi))h^*\\
&\simeq& 0.
\end{array}\] Writing $\delta\psi$ for this nullhomotopy, we thus have well-defined cobordisms\[\begin{array}{lccrl}
((&h&1&)\co &C\oplus C'\to C',(0,\psi\oplus-\psi')),\\
((&h(1-e)&-heg&)\co &C\oplus C'\to C',(\delta\psi,\psi\oplus-\psi')),
\end{array}\] and they are complementary as there exist the following left and right chain homotopy inverses\[\begin{array}{rcl}\left(\begin{array}{cc}eg&g\\ h(1-e)g&-1\end{array}\right)\left(\begin{array}{cc}h&1\\h(1-e)&-heg\end{array}\right)\simeq\left(\begin{array}{cc}1&0\\0&1\end{array}\right)&\co &C\oplus C'\to C\oplus C',\\
&&\\
\left(\begin{array}{cc}h&1\\h(1-e)&-heg\end{array}\right)\left(\begin{array}{cc}eg&g\\ h(1-e)g&-1\end{array}\right)\simeq \left(\begin{array}{cc}1&0\\0&1\end{array}\right)&\co &C'\oplus C'\to C'\oplus C'.\end{array}\proved\]\end{proof}

\begin{proposition}\label{prop:DLA}For $n\geq0$, double-cobordism is an equivalence relation on the set of homotopy equivalence classes of $n$--dimensional, $\eps$--ultraquadratic, Poincar\'{e} complexes over $R$. The equivalence classes form a group $\widehat{DL}_{n}(R,\eps)$, the \emph{$n$--dimensional, $\eps$--ultraquadratic double $L$--group of $R$}, with addition and inverses given by\[(C,\psi)+(C',\psi')=(C\oplus C',\psi\oplus\psi'),\qquad -(C,\psi)=(C,-\psi)\in \widehat{DL}_n(R,\eps).\]
\end{proposition}

\begin{proof}Lemma \ref{lem:welldef} shows in particular that double-cobordism is well-defined and reflexive. It is clearly symmetric. To show transitivity consider two double-cobordisms\[\begin{array}{lcl}
	c_\pm&=&((f_\pm\,\,f'_\pm)\co C\oplus C'\to D_\pm,(\delta_\pm\psi,\psi\oplus-\psi')),\\
	c'_\pm&=&((\tilde{f}'_\pm\,\,f''_\pm)\co C'\oplus C''\to D_\pm,(\delta_\pm\psi',\psi'\oplus-\psi'')).\\
\end{array}\]We intend to re-glue the 4 cobordisms according to the schematic in Figure \ref{fig:transitivity}.

\begin{figure}[h]
\def\pictransitivity{\resizebox{0.8\textwidth}{!}{ \includegraphics{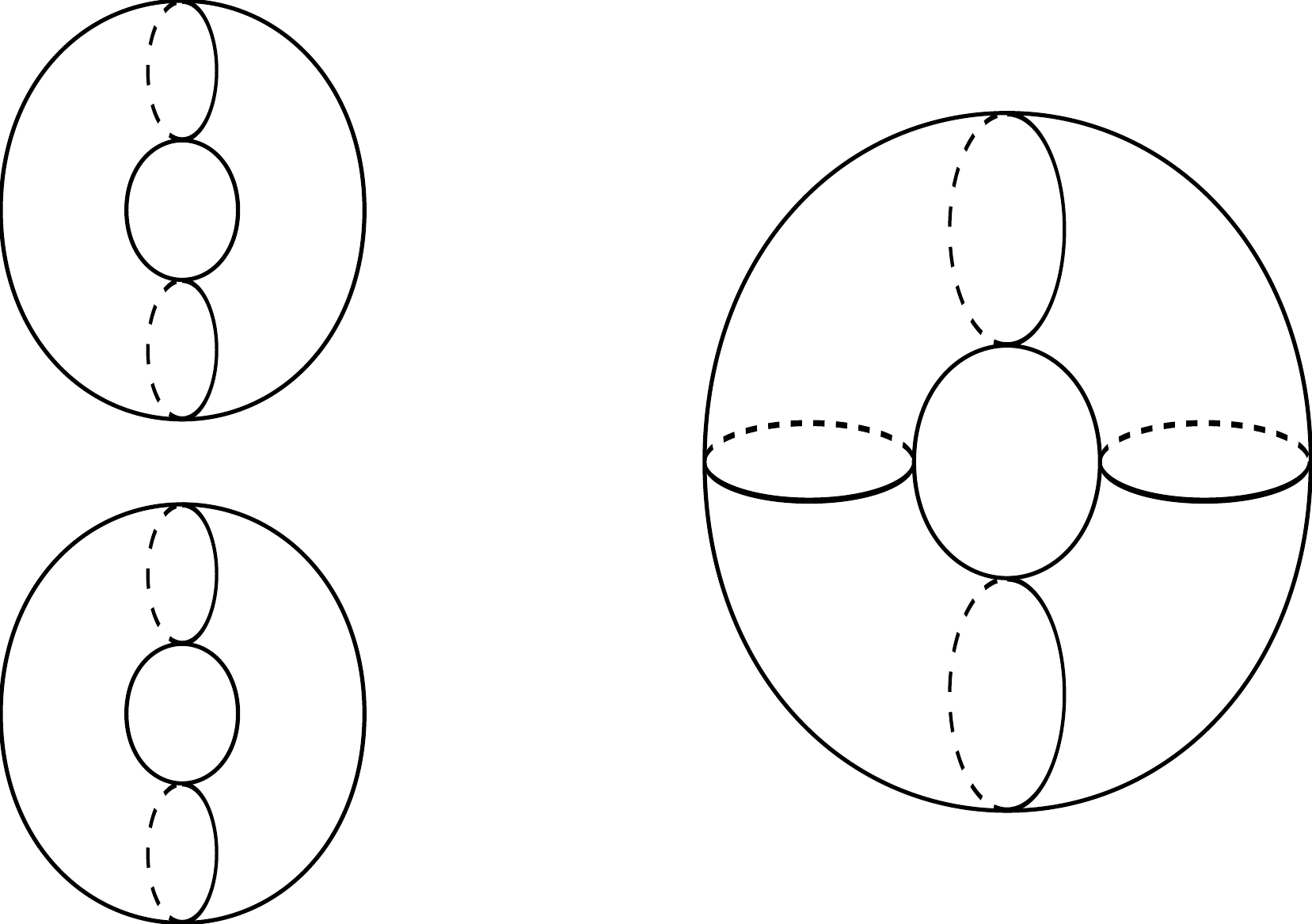}}}

\[\begin{xy} \xyimport(619.29,440){\pictransitivity}
,(82,405)*!L{C}
,(81,273)*!L{C'}
,(81,166)*!L{C'}
,(80,33)*!L{C''}
,(23,337)*!L{D_-}
,(133,337)*!L{D_+}
,(23,97)*!L{D_-'}
,(133,97)*!L{D_+'}
,(468,330)*!L{C}
,(468,110)*!L{C''}
,(373,220)*!L{C'}
,(560,220)*!L{C'}
,(390,290)*!L{D_-}
,(540,290)*!L{D_+}
,(390,140)*!L{D_-'}
,(540,140)*!L{D_+'}
,(180,390)*+{}="A";(360,340)*+{}="B"
,{"A"\ar@/^/"B"}
,(180,50)*+{}="C";(360,100)*+{}="D"
,{"C"\ar@/_/"D"}
\end{xy}\]
                \caption{Combining the double-nullcobordisms to show transitivity}
	       \label{fig:transitivity}
\end{figure}
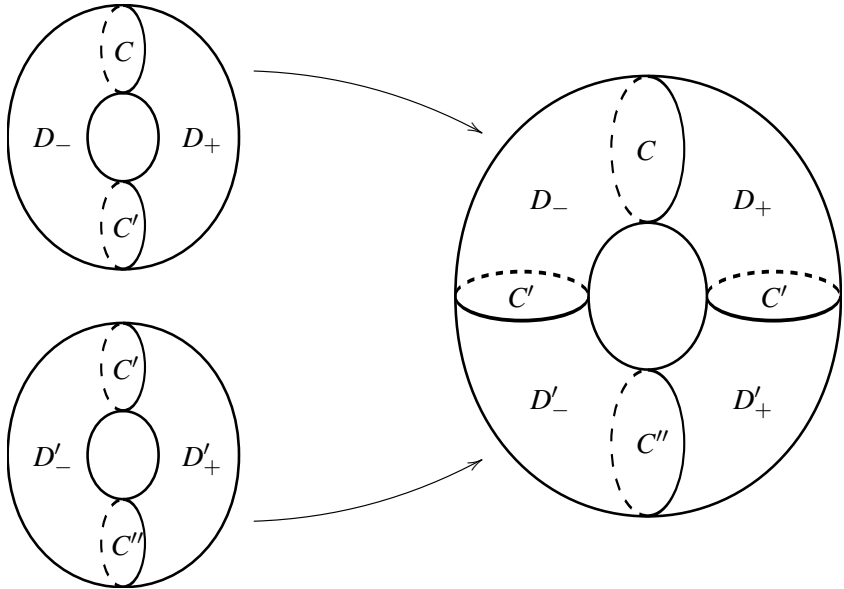

As $c_+$ and $c'_+$ share a boundary component, likewise $c_-$ and $c'_-$, we form the two algebraic unions \[c_\pm\cup c'_\pm=(C\oplus C''\to D_\pm'',(\delta_\pm\psi'',\psi\oplus-\psi''))\]where $D''_\pm=D_\pm\cup_{C'}D'_\pm$ is the mapping cone\[C\left({\left(\begin{array}{c}f'_\pm\\\tilde{f}'_\pm\end{array}\right)\co C'\to D_\pm\oplus D'_\pm}\right).\]To see that these two new cobordisms are complementary, first note that as our initial two double-cobordisms were complementary we have\[\left(\begin{array}{cc}f_+&f'_+\\f_-&f'_-\end{array}\right)\oplus\left(\begin{array}{cc}\tilde{f}'_+&f''_+\\\tilde{f}'_-&f''_-\end{array}\right)\co (C\oplus C')\oplus (C'\oplus C'')\xrightarrow{\simeq} (D_+\oplus D_-)\oplus (D'_+\oplus D'_-).\]But as $C\oplus C''$ is homotopy equivalent to the cone on the obvious inclusion $i\co C'\oplus C'\to C\oplus C'\oplus C'\oplus C''$ there is a homotopy commutative diagram with the map $g$ defined to make the left-hand square homotopy commute \[\xymatrix{
C'\oplus C' \ar[r]^-i\ar[d]^{=} & C\oplus C'\oplus C'\oplus C'' \ar[d]^{\simeq}\ar[r] &C(i)\simeq C\oplus C'' \ar[d]\\
C'\oplus C'\ar[r]^-g&D_+\oplus D_-\oplus D'_+\oplus D'_-\ar[r] &C(g)=D''_+\oplus D''_-}\] and hence the induced vertical map on the cones is a homotopy equivalence.
\end{proof}

\subsection*{Symmetric $DL$--groups}

For the symmetric case we work with the coefficients $A$, with $(A,S)$ defining a localisation.

\begin{definition}\label{def:comp2}For $n\geq 0$, two cobordisms between $n$--dimensional $\eps$--symmetric Poincar\'{e} complexes $(C,\phi)$ and $(C',\phi')$ \[x_\pm:=(f_\pm\co C\oplus C'\to D_\pm,(\delta_\pm\phi,\phi\oplus-\phi')\in C(f^\%)_{n+1})\]are \emph{complementary} if the chain map\[\left(\begin{smallmatrix}f_+\\f_-\end{smallmatrix}\right)\co C\oplus C'\to D_+\oplus D_-\]is a homotopy equivalence. In which case we say $(C,\phi)$ and $(C',\phi')$ are \emph{double-cobordant} and that the set $\{x_+,x_-\}$ is a \emph{double-cobordism} between them.

The $S$--acyclic versions of these definitions are made in exactly the same way but by restricting to the category $h\C_+(A,S)$ and using $(A,S)$--cobordisms.
\end{definition}

\begin{lemma}\label{lem:welldef2}For $n\geq0$, if $(C,\phi)$, $(C',\phi')$ are homotopy equivalent $n$--dimensional, $\eps$--symmetric Poincar\'{e} complexes over $A$, then they are double-cobordant. If they are $S$--acyclic they are $(A,S)$--double-cobordant.
\end{lemma}

\begin{proof}We work similarly to the proof of Lemma \ref{lem:welldef}. Let $h\co (C,\phi)\to(C',\phi')$ be a given homotopy equivalence. Recall our half-unit $s$ is assumed to be a central unit, so that in particular its action by multiplication on $A$--module chain complexes commutes with any chain map. The dual chain map to multiplication by $s$ is multiplication by the central unit $\overline{s}$. Calculate \[h\overline{s}\phi(h\overline{s})^*-s\phi's^*\simeq 0.\]Writing $\delta\phi$ for this nullhomotopy, we thus have well-defined cobordisms\[\begin{array}{lccrl}
((&h&1&)\co &C\oplus C'\to C',(0,\phi\oplus-\phi')),\\
((&h\overline{s}&-s&)\co &C\oplus C'\to C',(\delta\phi,\phi\oplus-\phi')).
\end{array}\]These cobordisms are easily calculated to be complementary as in Lemma \ref{lem:welldef}.
\end{proof}

\begin{proposition}\label{prop:DLAS}For $n\geq0$, double-cobordism is an equivalence relation on the set of homotopy equivalence classes of $n$--dimensional, $\eps$--symmetric, Poincar\'{e} complexes over $A$. With addition and inverses as in Proposition \ref{prop:DLA}, there is a well-defined group:\[DL^n(A,\eps):=\left\{\begin{array}{c}\text{double-cobordism classes of $n$--dimensional}\\ \text{$\eps$--symmetric Poincar\'{e} complexes}\end{array}\right\}.\]

Restricting to the category $h\C_+(A,S)$ and using $(A,S)$--cobordisms, there is similarly a well-defined group\[DL^n(A,S,\eps):=\left\{\begin{array}{c}\text{$(A,S)$--double-cobordism classes of $(n+1)$--dimensional}\\ \text{$S$--acyclic $(-\eps)$--symmetric Poincar\'{e} complexes}\end{array}\right\}.\]
\end{proposition}

\begin{proof}Exactly as in Proposition \ref{prop:DLA}.
\end{proof}

\subsection{Surgery above and below the middle dimension}\label{subsec:surgery}

We now turn to the question of calculating double $L$--groups for various rings and localisations. The only known programme for calculating chain complex bordism groups begins by proving that the groups are periodic in the dimension. Such periodicity is typically `skew 2--fold', and hence induces 4--fold periodicity -- for instance when $n\geq0$, Ranicki \cite{MR620795} shows that the following are isomorphic\[\widehat{L}_n(R,\eps)\cong \widehat{L}_{n+2}(R,-\eps), \quad L^n(A,\eps)\cong L^{n+2}(A,-\eps),\quad L^n(A,S,\eps)\cong L^{n+2}(A,S,-\eps),\] and hence each exhibits 4--fold periodicity. Skew-periodicity reduces the problem of calculation to the low-dimensional groups (those in dimension 0, 1). These groups can often be calculated in terms of more familiar tools such as Witt groups of forms or `formations' on f.g.\ projective or torsion modules (see \cite{MR620795}).

We will now investigate the extent to which this programme can be carried out for the $DL$--groups. We will prove periodicity in some restricted cases by a new technique called \emph{surgery above and below the middle dimension}, and in these cases reduce high-dimensional double $L$--theory to the low-dimensional case. In Section \ref{sec:DWgroups} we will relate the 0--dimensional double $L$--groups to what we called in \cite{OrsonA} the \emph{double Witt groups} and hence use the results of \cite{OrsonA} to calculate them for certain rings and localisations.

To begin the investigation, we will need another piece of technology from the literature, defined by Ranicki \cite{MR560997}. The \emph{skew-suspension} will allow comparison of structured chain complexes in different dimensions. Given a chain complex $C$ in $h\B_+(A)$, there is a homotopy equivalence defined by\[\overline{S}\co \Sigma^2(C^t\otimes_A C)\xrightarrow{\simeq} (\Sigma C)^t\otimes_A(\Sigma C);\qquad x\otimes y\mapsto (-1)^{|x|}x\otimes y\]in $\Ch(\Z[\Z/2\Z])$, where it is understood that the involution on $C^t\otimes_A C$ uses $T_\eps$ and the involution on $(\Sigma C)^t\otimes_A(\Sigma C)$ uses $T_{-\eps}$. Applying the $W^\%$ functor to this chain equivalence we obtain the homotopy equivalence\[\overline{S}\co \Sigma^{2}W_\eps ^\%C\xrightarrow{\simeq} W^\%_{-\eps}\Sigma C\]in $\Ch(\Z)$. There is also a relative version: given a morphism $f\co C\to D$ in $\B_+(A)$ there is a homotopy equivalence in $\Ch(\Z)$\[\overline{S}\co \Sigma^2C(f^\%\co W_\eps^\%C\to W^\%_\eps D)\xrightarrow{\simeq}C((\Sigma f)^\%\co W_{-\eps}^\%(\Sigma C)\to W^\%_{-\eps}(\Sigma D)).\]

\begin{definition}For $n\geq 0$, the \emph{skew-suspension} of an $n$--dimensional $\eps$--symmetric (Poincar\'{e}) complex $(C,\phi)$ is the $(n+2)$--dimensional, $(-\eps)$--symmetric (Poincar\'{e}) complex \[\overline{S}(C,\phi)=(\Sigma C,\overline{S}\phi).\]

The \emph{skew-suspension} of an $(n+1)$--dimensional $\eps$--symmetric $(f\co C\to D,(\delta\phi,\phi))$ is the $(n+3)$--dimensional, $(-\eps)$--symmetric (Poincar\'{e}) pair \[\overline{S}(f\co C\to D,(\delta\phi,\phi)):=(\Sigma f\co \Sigma C\to \Sigma D,\overline{S}(\delta\phi,\phi)).\]

The \emph{skew-suspension} of pairs and complexes in the $\eps$-ultraquadratic setting is defined similarly.
\end{definition}

\begin{proposition}\label{prop:skewsusp}For $n\geq 0$, the skew-suspension gives well-defined injective homomorphisms\[\begin{array}{rrclll}\overline{S}\co &\widehat{DL}_n(R,\eps)&\hookrightarrow &\widehat{DL}_{n+2}(R,-\eps);&\qquad&[(C,\psi)]\mapsto [\overline{S}(C,\psi)],\\
\overline{S}\co &DL^n(A,\eps)&\hookrightarrow &DL^{n+2}(A,-\eps);&\qquad&[(C,\phi)]\mapsto [\overline{S}(C,\phi)],\\
\overline{S}\co &DL^{n}(A,S,\eps)&\hookrightarrow &DL^{n+2}(A,S,-\eps);&\qquad&[(C,\phi)]\mapsto [\overline{S}(C,\phi)].\end{array}\]
\end{proposition}

\begin{proof}We consider only the statement for the groups $DL^n(A,\eps)$, the $S$--acyclic statement and the ultraquadratic statement being entirely similar.

If $(C,\phi)\in DL^n(A,\eps)$ admits complementary nullcobordisms $(f_\pm\co C\to D_\pm,(\delta_\pm\phi,\phi))$ then the skew-suspensions \[\overline{S}(f_\pm\co C\to D_\pm,(\delta_\pm\phi,\phi)):=(\Sigma f_\pm\co \Sigma C\to \Sigma D,\overline{S}(\delta\phi,\phi))\] are complementary nullcobordisms for $\overline{S}(C,\phi)\in DL^{n+2}(A,-\eps)$. Therefore the homomorphism is well defined.

To show injectivity, consider the general situation of a pair $x$ given by formal skew-desuspension an $(n+3)$--dimensional $(-\eps)$--symmetric pair \[x:=(\Sigma^{-1} f\co \Sigma^{-1} C\to \Sigma^{-1} D,(\overline{S})^{-1}(\delta\phi,\phi)),\]where $(\overline{S})^{-1}$ is a choice of homotopy inverse for $\overline{S}$. Such a formal skew-desuspension is an $(n+1)$--dimensional $\eps$--symmetric pair if and only if the morphism $\Sigma^{-1}f$ is in $h\B_+(A)$.

Now, more specifically, suppose that $(C,\phi)$ is an $n$--dimensional $\eps$--symmetric complex and that there are complementary nullcobordisms $(f_\pm\co \Sigma C\to D_\pm,(\Phi_\pm,\overline{S}\phi))$. Then the condition of being complementary gives $H_r(\Sigma C)\cong H_r(D_+)\oplus H_r(D_-)$ so that the desuspensions $\Sigma^{-1}D_\pm$ are in $h\B_+(A)$ because $H_r(\Sigma C)=0$ for $r\leq 0$. Hence $(\Sigma^{-1}f_\pm\co C\to \Sigma^{-1} D_\pm,(\overline{S})^{-1}(\Phi_\pm,\overline{S}\phi))$ are complementary nullcobordisms. The skew-suspension morphism $\overline{S}\co \Sigma^{2}W_\eps ^\%C\xrightarrow{\simeq} W^\%_{-\eps}\Sigma C$ is natural, and hence we also have that \[(\overline{S})^{-1}(\Phi_\pm,\overline{S}\phi)\simeq ((\overline{S})^{-1}\Phi_\pm,\phi)\in C((\Sigma^{-1}f)^\%)_{n+1}.\]Hence $(C,\phi)\sim 0\in DL^n(A,\eps)$ as required.

\end{proof}

We now move on to the cases in which we have been able to invert the skew-suspension map in double $L$--theory using a new technique called \emph{surgery above and below the middle dimension}. Our technique relies on Ranicki's concept of \emph{algebraic surgery} \cite[\textsection 4]{MR560997}. This is the process, in the setting of chain complexes with structure, which mimics geometric surgery on compact manifolds.

To describe algebraic surgery, begin with a possibly non-Poincar\'{e} $(n+2)$--dimensional $\eps$--symmetric (resp.\ $\eps$--ultraquadratic) pair \[x=(f\co C\to D,(\delta\phi,\phi))\qquad\text{(resp.\ $x=(f\co C\to D,(\delta\psi,\psi))$}\] with homotopy cofibration sequence \[\xymatrix{C\ar[r]^-{f}&D\ar[r]^-{e} &C(f)}.\]Recall we have $\lmat\delta\phi_0\\ \pm\phi_0 f^*\rmat\co D^{n+2-*}\to C(f)$ (in the $\eps$--ultraquadratic case, we use the symmetrisation of Definition \ref{def:symmetrisation}) and define $C':=\Sigma^{-1}C\left(\lmat\delta\phi_0\\ \pm\phi_0 f^*\rmat\right)$ so that there is a homotopy cofibration sequence\[\xymatrix{\dots\ar[r]&C'\ar[r]^-{f'}& D^{n+2-*}\ar[rr]^-{\lmat\delta\phi_0\\ \pm\phi_0 f^*\rmat}&& C(f)\ar[r]^-{e'}& \Sigma C'\ar[r]&\dots}\]with $f'$ the projection and $e'$ the inclusion. In the symmetric case Ranicki \cite[\textsection 4]{MR560997} showed that there is always a naturally defined $(n+2)$--dimensional $\eps$--symmetric pair\[x'=(f'\co C'\to D^{n+2-*},(\delta\phi',\phi')),\]such that there is a homotopy equivalence of Thom constructions\[(C(f),(\delta\phi/\phi))\simeq (C(f'),(\delta\phi'/\phi')).\]In the ultraquadratic case there are the following obstructions to building such a complex.

\begin{proposition}\label{prop:surgery}For an $(n+2)$--dimensional $\eps$--ultraquadratic pair $x=(f\co C\to D,(\delta\psi,\psi))$, the morphism $f':C'\to D^{n+2-*}$ (defined above) forms part of an $(n+2)$--dimensional $\eps$--ultraquadratic pair \[x'=(f'\co C'\to D^{n+2-*},(\delta\psi',\psi')),\]such that $(C(f),(\delta\psi/\psi))\simeq (C(f'),(\delta\psi'/\psi'))$, \emph{if and only if} the morphisms\[e'\circ(\delta\psi/\psi)_0\co C(f)^{n+2-*}\to C(f)\to \Sigma C',\quad e'\circ T_\eps(\delta\psi/\psi)_0\co C(f)^{n+2-*}\to C(f)\to \Sigma C',\]are nullhomotopic.
\end{proposition}

\begin{proof}First we make a general observation. Let $g:E\to F$ be any chain map and $h:F\to C(g)$ the corresponding map into the cone. Then as $h\circ g\simeq 0$, we have \[\begin{array}{rrcl}(h\otimes h)\circ(g\otimes g)\simeq 0:&E\otimes E&\to& C(g)\otimes C(g),\\
(1\otimes h)\circ(g\otimes g)\simeq 0:&E\otimes E&\to& F\otimes C(g),\\
(h\otimes 1)\circ(g\otimes g)\simeq 0:&E\otimes E&\to& C(g)\otimes F,\end{array}
\]and by the universal property of the mapping cone these nullhomotopies determine chain maps from $C(g\otimes g)$ to $C(g)\otimes C(g)$, $F\otimes C(g)$, and $C(g)\otimes F$ respectively. Apply this observation to the map $f'$ to obtain the three maps, which we call\[\begin{array}{rrcl}\theta:&C(f'\otimes f')&\to& C(f')\otimes C(f'),\\
\ev_r:&C(f'\otimes f')&\to& D^{n+2-*}\otimes C(f'),\\
\ev_l:&C(f'\otimes f')&\to& C(f')\otimes D^{n+2-*},\end{array}\]respectively (cf.\ definitions in Crowley-L\"{u}ck-Macko \cite[p.\ 341]{tibor}). Now, by definition of $f'$, we also have that $C(f)\simeq C(f')$, and so the homotopy class of the cycle $\delta\psi/\psi\in C(f)\otimes C(f)$ (which is determined by algebraic Thom construction on $x$) moreover determines a homotopy class of cycles in $C(f')\otimes C(f')$. Our task in this proof is to identify the obstruction to taking the homotopy class determined by $\delta\psi/\psi$ in $C(f')\otimes C(f')$ and lifting it to $C(f'\otimes f')$ along the map $\theta$. A choice of cycle in the lifted homotopy class would determine the $x'$ in the statement of the proposition.

To analyse the obstruction, we claim there is a homotopy pullback square\[ \xymatrix{C(f'\otimes f')\ar[r]^-{\lmat \ev_r\\ \ev_l\rmat}\ar[d]_-{\theta}&(D^{n+2-*}\otimes C(f'))\oplus (C(f')\otimes D^{n+2-*})\ar[d]^-{\left(\lmat\delta\phi_0\\ \pm\phi_0 f^*\rmat\otimes\id\right)\oplus\left(\id\otimes\lmat\delta\phi_0\\ \pm\phi_0 f^*\rmat\right)}\\C(f')\otimes C(f')\ar[r]^-{\lmat 1\\ T\rmat}&(C(f')\otimes C(f'))\oplus(C(f')\otimes C(f'))} \]The derivation of this square, and proof that it is a homotopy pullback, can be found in Crowley-L\"{u}ck-Macko \cite[Diagram (11.138) in the proof of Proposition 11.137]{tibor}, where the existence of such a square is proved for a general chain map in $h\B(R)$, which we take to be $f'$ (the reader is also referred to \cite[Diagram (11.99)]{tibor} and the subsequent discussion, from which many of the details of that proof are drawn). The cofibre of $\theta$ is given up to homotopy equivalence by the cone of the right hand column of the square, which is homotopy equivalent to $(\Sigma C'\otimes C(f))\oplus(C(f)\otimes \Sigma C')$. Hence, following $\delta\psi/\psi$ along the bottom row of the square, and into the cofibre of the right hand column, we obtain the obstructions described in the proposition.
\end{proof}

Using Proposition \ref{prop:surgery} we have been able to invert the skew-suspension map on double $L$--groups in the following algebraic setting.

\begin{theorem}[Surgery above and below the middle dimension]\label{thm:aboveandbelow}When $R$ has homological dimension 0 and $n\geq 0$, the skew-suspension defines isomorphisms:\[\begin{array}{rrcl}\overline{S}\co &\widehat{DL}_n(R,\eps)&\xrightarrow{\cong} &\widehat{DL}_{n+2}(R,-\eps),\\
\overline{S}\co &DL^n(A,\eps)&\xrightarrow{\cong} &DL^{n+2}(A,-\eps).\end{array}\]
\end{theorem}

\begin{proof}We will prove only the ultraquadratic case as the proof for the symmetric case is very similar but much simpler as there are no surgery obstructions there.

Because $R$ has homological dimension 0, all f.g.\ $R$--modules are objects of the f.g.\ \emph{projective} $R$--module category $\A(R)$. So for any morphism $f:P\to Q$ in $\A(R)$ there is a splitting $P\cong\im(f)\oplus \ker(f)$. It follows that any chain complex $(C_*,d)$ over $\A(R)$ is \emph{split} in the sense of Weibel \cite[\textsection 1.4.1]{MR1269324}, and hence there is a chain homotopy equivalence $(C_*,d)\simeq (H_*(C),0)$. So without loss of generality, we will assume an $(n+2)$--dimensional $(-\eps)$--ultraquadratic Poincar\'{e} complex $(C,\psi)$ over $R$ is of the form\[0\to C_{n+2}\xrightarrow{0} C_{n+1}\xrightarrow{0}\dots\xrightarrow{0} C_1\xrightarrow{0} C_0\to 0.\] Note this means the cochain complex $C^{n-*}$ also has all differentials 0. Consequently, the chain homotopy equivalence $\phi_0\co C^{n+2-*}\xrightarrow{\simeq}C$ moreover defines an isomorphism of projective $R$--modules $\phi_0\co C^{n+2-r}\xrightarrow{\cong}C_r$ for each $r$.

Define complexes $D_\pm$ and morphisms $f_\pm\co C\to D_\pm$ by\[\xymatrix{D_-&&0\ar[r]&0\ar[r]&0\ar[r]&\dots\ar[r]&0\ar[r]&C_0\ar[r]&0\\
C\ar[d]^{f_+}\ar[u]_{f_-}&&0\ar[r]&C_{n+2}\ar[r]\ar[d]^{1}\ar[u]_0&C_{n+1}\ar[r]\ar[u]_0\ar[d]^0&\dots\ar[r]&C_1\ar[r]\ar[r]\ar[u]_0\ar[d]^0&C_0\ar[r]\ar[r]\ar[u]_1\ar[d]^0&0\\
D_+&&0\ar[r]&C_{n+2}\ar[r]&0\ar[r]&\dots\ar[r]&0\ar[r]&0\ar[r]&0.}\]The differential in $C(f_\pm\otimes f_\pm)_*$ is given by\[d_\pm=\left(\begin{matrix}d_{D\otimes D}&(-1)^{r-1}(f_\pm\otimes f_\pm)\\0&d_{C\otimes C}\end{matrix}\right)\co (D\otimes D)_r\oplus (C\otimes C)_{r-1}\to (D\otimes D)_{r-1}\oplus (C\otimes C)_{r-2},\]and hence the elements $(0,\psi)\in C(f_\pm\otimes f_\pm)_{n+3}$ are moreover cycles as\[(d_\pm(0,\psi))_0=((-1)^{r-1}f\psi_0f^*)\oplus(d\psi_0)\co D^{n+2-r}\oplus C^{n-1-r}\to D_r\oplus C_r,\]which vanishes for all $r$ by inspection. So define two $(n+3)$--dimensional $(-\eps)$--ultraquadratic pairs $x_\pm=(f_\pm\co C\to D_\pm, (0,\psi))$. We must now check by hand that the obstructions to ultraquadratic algebraic surgery obtained in Proposition \ref{prop:surgery} vanish for surgery on $(C,\psi)$ with data $x_+$ and $x_-$ respectively. The two surgery obstructions for $x_+$ are given by the chain homotopy class of the composition\[C(f_+)^{n+3-*}\xrightarrow{\left(\delta\psi/\psi\right)_0} C(f_+)\xrightarrow{\text{incl.}}C\left(\left(\begin{matrix}\delta\phi_0 \\ \pm\phi_0f_+^*\end{matrix}\right):D_+^{n+3-*}\to C(f_+)\right)=:\Sigma C_+',\]and by its transposed version (see Proposition \ref{prop:surgery}). It may be calculated, as in \cite[p.\ 46]{MR620795}, that\[\left(\delta\psi/\psi\right)_0=\left(\begin{matrix}\delta\psi_0&0\\ (-1)^{n+3-r}\psi_0f_+^*&0\end{matrix}\right):D_+^{n+3-r}\oplus C^{n+2-r}\to (D_+)_r\oplus C_{r-1},\]so in fact the only possible non-zero map occurs where $r=1$. But $\Sigma C'$ is given by the following length $n+3$ complex\[\begin{array}{rcl}\Sigma C_+'&=&\xymatrix{C_{n+2}\ar[r]^-{\lmat 1\\ 0\rmat}&C_{n+2}\oplus C_{n+1}\ar[r]^-{0}&\dots\ar[r]^-{0}&C_1\oplus C^{n+2}\ar[r]^-{\lmat 0\\ \phi_0\rmat}&C_0\ar[r]^-{0}&0}\\
&\simeq&\xymatrix{0\ar[r]^-{0}&C_{n+1}\ar[r]^-{0}&\dots\ar[r]^-{0}&C_1\ar[r]^-{0}&0\ar[r]^-{0}&0}\end{array}\](where the homotopy equivalence uses that $\phi_0$ is an isomorphism). After this homotopy equivalence, the $r=1$ map has codomain 0 and hence the first surgery obstruction vanishes homotopically. For the second surgery obstruction, note that the matrix for $T\left(\delta\psi/\psi\right)_0$ is given by transposing the matrix for $\left(\delta\psi/\psi\right)_0$ and applying the $\eps$--involution $T_\eps$ to each entry. Hence the only possible non-trivial morphism occurs now when $r=n+2$. But actually, inspection of the same chain homotopy representative of $\Sigma C_+'$ shows that this morphism also vanishes. Therefore there is no obstruction to the $x_+$ surgery. A very similar argument shows that the $x_-$ surgery is also unobstructed and we leave this check to the reader.

Both effects of surgery $(C'_\pm,\psi'_\pm)$ are given, up to homotopy, by \[(C'_\pm)_r=\left\{\begin{array}{lcl}C_r&\quad&1\leq r\leq n+1,\\0&&\text{otherwise,}\end{array}\right. \]with all differentials equal to 0 and \[(\psi'_+)_0=(\psi'_-)_0 =\left\{\begin{array}{lcl}\psi_0\co C^{n+2-r}\to C_r&&1\leq r\leq n+1,\\0&&\text{otherwise.}\end{array}\right. \]Hence there exists a homotopy equivalence $h\co (C_-',\psi_-')\xrightarrow{\simeq}(C_+',\psi_+')$. Now according to Ranicki \cite[Proposition 4.1 (ii)]{MR560997} we may write cobordisms \[((g_\pm\,\,g'_\pm)\co C\oplus C_\pm'\to \tilde{D}_\pm,(0,\psi\oplus-\psi'_\pm))\] whose underlying morphisms are given on the level of homology by \[
(g_+\,\,g'_+)_*=\left\{\begin{array}{ccl}(1\,\, 1)\co H_r(C)\oplus H_r(C)\to H_{r}(C)&\quad&1\leq r\leq n+1,\\ 1\co H_r(C)\to H_{r}(C)&&r=0\\0&&\text{otherwise}\end{array}\right.
\]\[
(g_-\,\,g'_-)_*=\left\{\begin{array}{ccl}(1\,\, 1)\co H_r(C)\oplus H_r(C)\to H_{r}(C)&\quad&1\leq r\leq n+1,\\ 1\co H_r(C)\to H_{r}(C)&&r=n+2\\0&&\text{otherwise}\end{array}\right.
\]We must modify these to be \emph{complementary} cobordisms. As in the proof of Lemma \ref{lem:welldef}, and as shown below, we may modify one of the cobordisms using $e:=\psi\phi_0^{-1}$ and one using the homotopy $h$. Modifying the morphisms by $e$ may result in the element $(0,\psi,\oplus-\psi'_-)$ no longer being a cycle. But as in the proof of Lemma \ref{lem:welldef}, $(1-e)g_-\psi g_-^*(1-e)^*-eg_-'\psi'_-(g_-')^*e^*$ is calculated to be nullhomotopic, and we denote a choice of nullhomotopy by $\Psi$. Define\[\begin{array}{lccrlr}((&g_+&g_+'h&)\co &C\oplus C'_-\to \tilde{D}_+,&(0,\psi\oplus-\psi'_-)),\\((&(1-e)g_-&-eg_-'&)\co &C\oplus C'_-\to \tilde{D}_-,&(\Psi,\psi\oplus-\psi'_-)).\end{array}\]Working exactly as in Lemma \ref{lem:welldef}, the morphism \[\left(\begin{matrix}g_+&g_+'h\\(1-e)g_-&-eg'_-\end{matrix}\right)\co C\oplus C_-'\to \tilde{D}_+\oplus \tilde{D}_-\]is seen to be an isomorphism on the level of homology, which is sufficient to show chain equivalence as we are working with bounded complexes. Hence we have shown that $(C,\psi)$ is double-cobordant to the skew-suspension of the $n$--dimensional $\eps$--ultraquadratic Poincar\'{e} complex given by $(\Sigma^{-1}C_-',\overline{S}^{-1}\psi_-')$. In other words, the skew suspension map $\overline{S}\co \widehat{DL}_n(R,\eps)\to \widehat{DL}_{n+2}(R,-\eps)$ is surjective and therefore an isomorphism as was required to be shown.
\end{proof}

\begin{corollary}\label{cor:aboveandbelow}Under the hypotheses of Theorem \ref{thm:aboveandbelow} and for $k\geq 0$ \[\begin{array}{ll}\widehat{DL}_{2k+1}(R,\eps)=0, &\widehat{DL}_{2k}(R,\eps)\cong\widehat{DL}_0(R,(-1)^k\eps),\\ DL^{2k+1}(A,\eps)=0,&DL^{2k}(A,\eps)\cong DL^0(A,(-1)^k\eps).\end{array} \]
\end{corollary}
\begin{proof}In each case, apply the isomorphism of Theorem \ref{thm:aboveandbelow} $k$ times to obtain an isomorphism to either a 1--dimensional or 0--dimensional double $L$--group. For the even dimensional case, we are now done.

For any 1--dimensional $(-1)^k\eps$--ultraquadratic Poincar\'{e} complex $(C,\psi)$, the pairs $(f_\pm\co C\to D_\pm,(0,\psi))$ defined in the proof of Theorem \ref{thm:aboveandbelow} are complementary nullcobordisms, proving that $DL^1(R,(-1)^k\eps)=0$. The 1--dimensional $(-1)^k\eps$--symmetric case is entirely similar.
\end{proof}

\section{Double Witt groups and double $L$--groups}\label{sec:DWgroups}

In this section we use hyperbolic versions of the classical torsion Witt groups of algebraic number theory to cast the low-dimensional double $L$--groups in a form more amenable to calculation. Based on our results in \cite{OrsonA} and our surgery results of Section \ref{sec:DLgroups} we are able to completely calculate the double $L$--groups $\widehat{DL}_n(R,\eps)$ for $n\geq0$ when $R$ has homological dimension 0. Furthermore, we will derive some new algebraic results relating to Seifert forms and for linking forms using the techniques of double $L$--theory. These will be applied in Section \ref{sec:knots} to high-dimensional knot theory, in particular to give algebraic answers to Question \ref{q:stablyhypishyp}.

\subsection{Forms and linking forms}

The language we use for forms and linking forms is based on that found in Ranicki \cite[\textsection\textsection  1.6, 3.4]{MR620795} although we caution that our terminology, particularly later on regarding lagrangian submodules, differs slightly. Also, our use of the word `split' in reference to forms and linking forms is entirely different to Ranicki's use.

\begin{definition}An \emph{$\eps$--symmetric form over $A$} is a pair $(P,\theta)$ consisting of a f.g.\ projective $A$--module
 $P$ and an injective $A$--module morphism $\theta\co P\hookrightarrow P^*$ such that $\theta(x)(y)=\eps\overline{\theta(y)(x)}$ for all $x,y\in P$ (equivalently $\theta=\eps\theta^*$). A form $(P,\theta)$ is \emph{non-singular} if $\theta$ is an isomorphism. A form induces a sesquilinear pairing also called $\theta$\[\theta\co P\times P\to A;\qquad (x,y)\mapsto \theta(x,y):=\theta(x)(y).\]A morphism of $\eps$--symmetric forms $(P,\theta)\to (P',\theta')$ is an $A$--module morphism $f\co P\to P'$ such that $\theta(x)(y)=\theta'(f(x))(f(y))$ (equivalently $\theta=f^*\theta' f$), it is an isomorphism when $f$ is an $A$--module isomorphism. The set of isomorphism classes of $\eps$--symmetric forms over $A$, equipped with the addition $(P,\theta)+(P',\theta')=(P\oplus P',\theta\oplus \theta')$ forms a commutative monoid \[\NN^\eps(A)=\{\text{$\eps$--symmetric forms over $A$}\}.\]
\end{definition}

\begin{definition}\label{def:linking}Suppose $(A,S)$ defines a localisation. An \emph{$\eps$--symmetric linking form over $(A,S)$} is a pair $(T,\lambda)$ consisting of an object $T$ of $\H(A,S)$ and an injective $A$--module morphism $\lambda\co T\hookrightarrow T^\wedge$ such that $\lambda(x)(y)=\eps\overline{\lambda(y)(x)}$ for all $x,y\in T$ (equivalently $\lambda=\eps\lambda^\wedge$). A linking form $(T,\lambda)$ is \emph{non-singular} if $\lambda$ is an isomorphism. A linking form induces a sesquilinear pairing also called $\lambda$\[\lambda\co T\times T\to S^{-1}A/A;\qquad (x,y)\mapsto \lambda(x,y):=\lambda(x)(y).\]A morphism of $\eps$--symmetric linking forms $(T,\lambda)\to (T',\lambda')$ is an $A$--module morphism $f\co T\to T'$ such that $\lambda(x)(y)=\lambda'(f(x))(f(y))$ (equivalently $\lambda=f^\wedge\lambda' f$). $f$ is an isomorphism of forms when $f$ is an $A$--module isomorphism. The set of isomorphism classes of $\eps$--symmetric linking forms over $A$, equipped with the addition $(T,\lambda)+(T',\lambda')=(T\oplus T',\lambda\oplus \lambda')$ forms a commutative monoid \[\NN^\eps(A,S)=\{\text{$\eps$--symmetric linking forms over $(A,S)$}\}.\]
\end{definition}

\subsection{Double Witt groups}\label{sec:witt}

We will now define several ways in which an $\eps$--symmetric form or linking form can be considered trivial, that all involve the idea a lagrangian submodule.

\begin{definition}A \emph{lagrangian} for a non-singular $\eps$--symmetric form $(P,\theta)$ over $A$ is a submodule $j \co  L \hookrightarrow P$ in $\A(A)$ such that the sequence\[0\to L\xrightarrow{j}P\xrightarrow{j^*\theta} L^*\to0\]is exact. As modules in the category $\A(A)$ are projective, all surjective morphisms split, and a lagrangian is always a direct summand. If $(P,\theta)$ admits a lagrangian it is called \emph{metabolic}. If $(P,\theta)$ admits two lagrangians $j_\pm\co L_\pm\hookrightarrow P$ (labelled `$+$' and `$-$') such that they are complementary as submodules\[\left(\begin{matrix}j_+\\j_-\end{matrix}\right)\co L_+\oplus L_-\xrightarrow{\cong} P,\]then the form is called \emph{hyperbolic}.

If $(A,S)$ defines a localisation then the definitions of \emph{(split) lagrangian}, \emph{(split) metabolic} and \emph{hyperbolic} transfer immediately to the setting of non-singular $\eps$-symmetric linking forms, but now using the torsion dual $-^\wedge$ instead of the projective dual $-^*$.
\end{definition}

In general, not assuming the presence of a half-unit, for a symmetric \textbf{linking form} there is a hierarchy:\[\text{hyperbolic}\,\,\subsetneq\,\,\text{split metabolic}\,\,\subsetneq\,\,\text{metabolic}.\]And for a \textbf{form}\[\text{hyperbolic}\,\,\subsetneq\,\,\text{split metabolic}\,\,=\,\,\text{metabolic}.\]It is a standard result that the presence of a half-unit destroys the distinction between split metabolic and hyperbolic both for forms and for linking forms. More generally, if the form or linking form admits a quadratic extension, this distinction is destroyed.

\begin{definition}[Monoid construction]\label{def:monoid} Let $(M, +)$ be an commutative monoid and let $N$ be a submonoid of $M$. Consider the equivalence relation: for $m_1, m_2 \in M$, define $m_1 \sim m_2$ if there exists $n_1, n_2 \in N$ such that $m_1 + n_1 = m_2 + n_2$. Then the set of equivalence classes $M/\sim$ inherits a structure of abelian monoid via $[m] + [m'] := [m + m']$. It is denoted by $M/N$. Assume that for any element $m \in M$ there is an element $m' \in M$ such that $m + m' \in N$, then $M/N$ is an abelian group with $-[m] = [m']$.
\end{definition}

When the monoid construction returns a group, it is in general the group of \emph{stable isomorphism classes} in $M$, where one is allowed to `stabilise an isomorphism', on either side, by elements of the submonoid.

\begin{lemma}If $(P,\theta)$ is an $\eps$--symmetric form over $A$ then the form $(P\oplus P,\theta\oplus-\theta)$ is split metabolic. If moreover there exists a half-unit $s\in A$, then $(P\oplus P,\theta\oplus-\theta)$ is hyperbolic.
\end{lemma}

\begin{proof}The diagonal $\lmat 1\\1\rmat\co P\to P\oplus P$ is a lagrangian with splitting $(1\,\,0) \co P\oplus P\to P$.

If $s$ is a half-unit, then $\lmat \overline{s} \\ -s \rmat\co P\oplus P$ is a lagrangian, and it is complementary to the diagonal as \[\left(\begin{matrix}s & 1\\\bar{s} & -1\end{matrix}\right)\left(\begin{matrix}1&1\\\bar{s}& -s\end{matrix}\right)=\left(\begin{matrix}1&1\\\bar{s}& -s\end{matrix}\right)\left(\begin{matrix}s & 1\\\bar{s} & -1\end{matrix}\right)=I.\proved\]\end{proof}

The preceding lemma holds completely analogously for linking forms, justifying the following definitions:

\begin{definition}Suppose $(A,S)$ defines a localisation. The monoid constructions\[\begin{array}{rcl}
W^\eps(A)&=&\NN^\eps(A)/\{\text{metabolic forms}\}\\
W^\eps(A,S)&=&\NN^\eps(A,S)/\{\text{metabolic linking forms}\}
\end{array}\]are abelian groups called the \emph{$\eps$--symmetric Witt group} of $A$ and of $(A,S)$ respectively. The monoid construction\[\begin{array}{rcl}
DW^\eps(A,S)&=&\NN^\eps(A,S)/\{\text{hyperbolic linking forms}\}
\end{array}\] is an abelian group called the \emph{$\eps$--symmetric double Witt group} of $(A,S)$.
\end{definition}

\subsection{The 0--dimensional double $L$--groups as double Witt groups}

A 0--dimensional $\eps$--symmetric (Poincar\'{e}) complex $(C,\phi)$ over $A$ is the same thing as a (non-singular) $\eps$--symmetric pairing on $H^0(C)$. A 1--dimensional $\eps$--symmetric $S$--acyclic (Poincar\'{e}) complex $(C,\phi)$ has an interpretation as a torsion pairing on $H^1(C)$. Cobordant Poincar\'{e} complexes correspond to Witt equivalent non-singular forms and linking forms, and the low-dimensional $L$--groups $L^0(A,\eps)$ and $L^0(A,S,\eps)$ are well-known to be isomorphic to Witt groups of forms and of linking forms respectively (see Ranicki \cite{MR620795}).

The double Witt groups are to the low-dimensional double $L$--groups as the Witt groups are to the low-dimensional $L$--groups. In this Subsection we will prove the double Witt groups are indeed isomorphic to the 0--dimensional double $L$--groups (see Propositions \ref{prop:isogroups} and \ref{iso}). Unexpectedly, the proof of this is different to the proof of the corresponding fact in classical $L$--theory.

\subsection*{Ultraquadratic double $L$--groups and Seifert forms}

As discussed in Subsection \ref{sec:witt}, when there is a half-unit in the coefficient ring, there is no difference between hyperbolic and metabolic forms on projective modules. However, without assuming there is a half-unit, the 0--dimensional double $L$--groups $\widehat{DL}_0(R,\eps)$ have a good interpretation in terms of \emph{Seifert} forms.

\begin{definition}An \emph{$\eps$--symmetric Seifert form} $(K,\psi)$ over $R$ is a f.g.\ projective $R$--module $K$ and a morphism of $R$--modules $\psi\co K\to K^*$ such that $\psi+\eps\psi^*$ is an isomorphism (note that this makes $(K,\psi+\eps\psi^*)$ a non-singular $\eps$--symmetric form). For a Seifert form $(K,\psi)$ we define an endomorphism $e=(\psi+\eps\psi^*)^{-1}\psi$  and we note that this gives $1-e=\eps(\psi+\eps\psi^*)^{-1}\psi^*$). A morphism of Seifert forms $g\co (K,\psi)\to(K',\psi')$ is a morphism of $R$--modules $g\co K\to K'$ such that $g^*\psi'g=\psi$, it is an isomorphism if $g$ is an $R$--module isomorphism.
\end{definition}

The following is then easily checked from the definitions:

\begin{proposition}\label{prop:correspondenceproj}Sending $(C,\psi)\mapsto (H^0(C),\psi)$ defines a contravariant equivalence of categories that preserves the monoid structure:

\[\begin{array}{rcl}\left\{  \begin{array}{c}\text{0--dimensional,}\\ \text{$\eps$--ultraquadratic,} \\ \text{Poincar\'{e} complexes over $R$}\end{array} \right\}_{\text{\big{/htpy.}}}&\longleftrightarrow &
\left\{  \begin{array}{c}\text{non-singular $\eps$--symmetric} \\ \text{Seifert forms over $R$}\end{array} \right\}_{\text{\big{/iso}}}.\end{array}\]
\end{proposition}

The action of $e$ on the underlying $R$--module $K$ of a Seifert form $(K,\psi)$ makes $K$ an $R[s]$--module where $s$ is a formal variable with action $s(x):=e(x)$ for $x\in K$. $R[s]$ is a ring with involution where we extend the involution from $R$ by $\overline{s}=1-s$. When we wish to remember the morphism by which $s$ acts, we will write $(K,e)$. The \emph{Seifert dual} of an $R[s]$--module $(K,e)$ is the $R[s]$--module $(K,e)^*=(K^*=\Hom_R(K,R),1-e^*)$. An $R[s]$--submodule of $K$ is an $R$--submodule $j\co L\hookrightarrow K$ such that $ej(L)\subset j(L)$. Such an $L$ inherits an $R[s]$--module structure in the obvious way.

\begin{definition}\label{def:DWseifert}A \emph{(split) lagrangian} for a $\eps$--symmetric Seifert form $(K,\psi)$ over $R$ is a $R[s]$--submodule $j \co  L \hookrightarrow K$ such that the sequence in the category of $R[s]$--modules and Seifert duals\[0\to (L,e)\xrightarrow{j}(K,e)\xrightarrow{j^*(\psi+\eps\psi^*)} (L,e)^*\to0\]is (split) exact, and $j^*\psi j=0$. If $(K,\psi)$ admits a (split) lagrangian it is called \emph{(split) metabolic}. If $(K,\psi)$ admits two lagrangians $j_\pm\co L_\pm\hookrightarrow K$ such that they are complementary as $R[s]$--submodules\[\left(\begin{matrix}j_+\\j_-\end{matrix}\right)\co L_+\oplus L_-\xrightarrow{\cong} K,\]then the Seifert form is called \emph{hyperbolic}. We denote the corresponding \emph{$\eps$--symmetric Witt and double Witt group of Seifert forms over $R$} respectively by\[\widehat{W}_\eps(R)\qquad\text{and}\qquad \widehat{DW}_\eps(R).\]
\end{definition}

Here is a precise characterisation of metabolic Seifert forms considered as chain complexes.

\begin{proposition}\label{prop:metcxseif}An $\eps$--symmetric Seifert form over $R$ admits a lagrangian if and only if the associated homotopy equivalence class of 0--dimensional $\eps$--ultraquadratic Poincar\'{e} complexes over $R$ contains $(C,\psi)$ such that there is a nullcobordism $(f\co C\to D,(\delta\psi,\psi))$ with $H^1(D)=0$.
\end{proposition}

\begin{proof}An $\eps$--symmetric Seifert form $(K,\psi)$ corresponds to the $\eps$--ultraquadratic chain complex $(C,\psi)$ where $C$ is concentrated in degree 0 with $C_0=K^*$.

A lagrangian $j\co L\hookrightarrow K$ for $(K,\psi)$ defines a morphism of chain complexes concentrated in degree 0 \[f=j^*\co C\to D,\qquad D_0:=L^*.\]The morphism $\psi\co K\to K^*$ determines a 0--cycle $\psi_0\in \Hom(C^{-*},C)_0\simeq(C\otimes C)_0$, which has $(f\otimes f)(\psi)=0\in (D\otimes D)_0$ as $j^*\psi j=0$. Hence, considering the Puppe sequence of $f\otimes f$, we may lift $\psi$ to a 1--dimensional $\eps$--ultraquadratic structure $(0,\psi)$. So there is a 1--dimensional $\eps$--ultraquadratic Poincar\'{e} pair $(f\co C\to D,(0,\psi))$, and clearly $H^1(D)=0$.

Conversely, suppose there is a 1--dimensional $\eps$--ultraquadratic Poincar\'{e} pair $(f\co C\to D,(\delta\psi,\psi))$ with $0=H^1(D)=H_0(D,C)$. Then there is a short exact sequence \[0\to H^0(D)\to D^0\to D^1\to 0,\] and $D^0, D^1$ are projective so we have that $H^0(D)$ is a f.g.\ projective $R$--module as well. Dualising the short exact sequence, we see that $0=H_1(D)=H^0(D,C)$ and $H^0(D)^*\cong H_0(D)$. It is now standard to check that we have a lagrangian of the non-singular $\eps$--symmetric form $(H^0(C),\psi+\eps\psi^*)$ over $R$, given by \begin{equation}\label{eq:lagrangian}\xymatrix{0\ar[r]& L=H^0(D)\ar[rr]^-{f^*}&& H^0(C)\ar[rr]^-{f(\psi+\eps\psi^*)}&& L^*\cong H_0(D)\ar[r] &0.}\end{equation}We wish to show this is moreover a lagrangian of the Seifert form $(H^0(C),\psi)$. We must first check that $L$ is an $R[s]$--submodule, in other words that when $x\in H^0(D)$ then $ef^*(x)\in\im(f^*)$. Using the Poincar\'{e} duality isomorphisms, this is equivalent to checking that $\psi f^*(x)\in\ker(f\co H_0(C)\to H_0(D))$. But $f\psi f^*$ vanishes on $H^*(D)$ as $(f\co C\to D, (\delta\psi,\psi))$ is Poincar\'{e}. So $e(L)\subset L$ and $j\co L\hookrightarrow K$ is an $R[s]$--module morphism. Moreover, the reader may check that the action of $s$ on the Seifert dual of $L$ commutes with the map $f(\psi+\eps\psi^*)\co H^0(C)\to H_0(D)$ of diagram \ref{eq:lagrangian}, so that this is a sequence of $R[s]$--modules. We finally note that the kernels and images in the sequence of $R$--modules are the same (setwise) when it is a sequence of $R[s]$--modules, so the sequence is still exact.
\end{proof}

Hence a precise characterisation of hyperbolic Seifert forms considered as chain complexes:

\begin{proposition}\label{prop:hypcxseif}An $\eps$--symmetric Seifert form over $R$ is hyperbolic if and only if the associated $\eps$--symmetric homotopy equivalence class of 0--dimensional $\eps$--ultraquadratic Seifert complexes over $R$ contains $(C,\psi)$ such that there are two complementary nullcobordisms $(f_\pm\co C\to D_\pm,(\delta\psi,\psi))$ with $H^1(D_\pm)=0$.
\end{proposition}

We now finally obtain an alternative characterisation of the 0--dimensional double $L$--groups of ultraquadratic forms.

\begin{proposition}\label{prop:isogroups}There is an isomorphism of groups\[
\widehat{DW}_\eps(R)\xrightarrow{\cong} \widehat{DL}_0(R,\eps).\]
\end{proposition}

\begin{proof}The morphism is defined in Proposition \ref{prop:correspondenceproj}. It is well-defined and surjective by the Propositions \ref{prop:correspondenceproj}, \ref{prop:metcxseif} and \ref{prop:hypcxseif}. To show injectivity, suppose that if $(C,\psi)$ is a $\eps$--ultraquadratic, 0--dimensional Poincar\'{e} complex associated to the $\eps$--symmetric Seifert form $(K,\psi)$. If there exists a pair of complementary Seifert nullcobordisms $(f_\pm\co C\to D_\pm,(\delta_\pm\psi,\psi))$ then in particular $0=H^1(C)=H^1(D_+)\oplus H^1(D_-)$ so that $H^1(D_\pm)=0$. But then $(K,\psi)$ must be hyperbolic by Proposition \ref{prop:hypcxseif}.
\end{proof}

We are now in a position to give some sample calculations of double $L$--groups.

\begin{example}\label{ex:calculation}We may apply Corollary \ref{cor:aboveandbelow} and the calculations in the author's paper \cite[Example 4.9]{OrsonA} to obtain calculations of the double $L$--groups when $R$ is a field of characteristic not 2. These calculations are made in terms of involution invariant prime ideals $\p$ of the Laurent polynomial ring $R[z,z^{-1}]$ which do not augment to 0 under $z\mapsto 1$. For example, when $R=\C$, such $\p$ are generated by polynomials $z-a$ where $a\in S^1\sm\{1\}\subset\C$ and in \cite{OrsonA} we calculated:\[\widehat{DL}_n(\C,\eps)\cong\left\{\begin{array}{lcl}\bigoplus_{a\in S^1\sm\{1\}}\bigoplus_{l=1}^\infty\Z&&\text{$n$ even,}\\ 0&&\text{$n$ odd,}\end{array}\right.\]When $R=\R$, such $\p$ are generated by polynomials $z^2-2cos\theta+1$ with $0<\theta<\pi$. We calculated:\[\widehat{DL}_n(\R,\eps)\cong\left\{\begin{array}{lll}\bigoplus_{0<\theta<\pi}\bigoplus_{l=1}^\infty\Z&&\text{$n=2k$ and $(-1)^k\eps=1$,}\\
(\bigoplus_{l=1}^\infty\Z)\oplus(\bigoplus_{0<\theta<\pi}\bigoplus_{l=1}^\infty\Z)&&\text{$n=2k$ and $(-1)^k\eps=-1$,}\\
0&&\text{$n$ odd.}\end{array}\right.\]

\end{example}

A non-singular $\eps$--symmetric Seifert form $(K,\psi)$ over $R$ is called \emph{stably hyperbolic} if there exist hyperbolic $\eps$--symmetric Seifert forms $H,H'$ such that $(K,\psi)\oplus H\cong H'$. Note that \emph{a priori} the stably hyperbolic Seifert forms are precisely the representatives of the 0 class in the double Witt group of Seifert forms. However, we obtain the following (new) characterisation as a corollary of Proposition \ref{prop:isogroups}.

\begin{corollary}[`Stably hyperbolic $=$ hyperbolic']\label{cor:stablyhypishypseif}A non-singular $\eps$--symmetric Seifert form $(K,\psi)$ over $R$ is hyperbolic if and only if it is stably hyperbolic.
\end{corollary}

\begin{proof}`Only if' is clear. Conversely, a stably hyperbolic $\eps$--symmetric Seifert form determines the 0 class in $\widehat{DW}_\eps(R)\cong \widehat{DL}_0(R,\eps)$. Therefore there is a double-nullcobordism of the corresponding 0--dimensional $\eps$--ultraquadratic complex. But by the proof of Proposition \ref{prop:isogroups}, these nullcobordisms correspond to complementary lagrangians.
\end{proof}

\subsection*{Double Witt groups of linking forms}

We now perform a very similar analysis for linking forms.

\begin{proposition}[{Ranicki \cite[3.4.1]{MR620795}}]\label{prop:correspondence}There is a contravariant equivalence of categories that preserves the monoid structure:
\[\begin{array}{rcl}\left\{  \begin{array}{c}\text{1--dimensional, $(-\eps)$--symmetric} \\ \text{$S$--acyclic (Poincar\'{e})}\\ \text{complexes over $A$}\end{array} \right\}_{\text{\big{/htpy.}}}&\longleftrightarrow &
\left\{  \begin{array}{c}\text{(non-singular)}\\ \text{$\eps$--symmetric linking} \\ \text{forms over $(A,S)$}\end{array} \right\}_{\text{\big{/iso}}} \end{array}\]which sends $(C,\phi)\mapsto (H^1(C),\lambda_\phi)$, where $\lambda_\phi([x],[y])=s^{-1}\phi_0(x)(z)$ for $x,y\in C^1$, $z\in C^0$ and $s\in S$ such that $d^*z=sy$.
\end{proposition}

Here is a precise characterisation of metabolic forms considered as chain complexes:

\begin{proposition}[{Ranicki \cite[3.4.5(ii)]{MR620795}}]\label{prop:3.4.5}A non-singular, $\eps$--symmetric linking form over $(A,S)$ admits a lagrangian if and only if the associated $(-\eps)$--symmetric homotopy equivalence class of 1--dimensional $(-\eps)$--symmetric Poincar\'{e} complexes over $A$ contains $(C,\phi)$ such that there is an $S$--acyclic 2--dimensional $(-\eps)$--symmetric Poincar\'{e} pair $(f\co C\to D,(\delta\phi,\phi))$ with $H^2(D)=0$.
\end{proposition}

Hence a precise characterisation of hyperbolic forms considered as chain complexes:

\begin{proposition}\label{hypcx}A non-singular, $\eps$--symmetric linking form over $(A,S)$ is hyperbolic if and only if the associated $(-\eps)$--symmetric homotopy equivalence class of 1--dimensional $(-\eps)$--symmetric Poincar\'{e} complexes over $A$ contains $(C,\phi)$ such that there are two complementary $S$--acyclic 2--dimensional $(-\eps)$--symmetric pairs $(f_\pm\co C\to D_\pm,(\delta_\pm\phi,\phi))$ with $H^2(D_\pm)=0$.
\end{proposition}

\begin{proposition}\label{iso}Let $A$ be a ring with involution which contains a half-unit, then there is an isomorphism of groups\[DW^\eps(A,S)\xrightarrow{\cong} DL^0(A,S,\eps).\]
\end{proposition}
\begin{proof}The morphism is defined in Proposition \ref{prop:correspondence}. It is well-defined and surjective by the Propositions \ref{prop:correspondence}, \ref{prop:3.4.5} and \ref{hypcx}. To show injectivity, suppose that if $(C,\phi)$ is a $(-\eps)$--symmetric, 1--dimensional Poincar\'{e} complex associated to an $\eps$--symmetric linking form $(T,\lambda)$. If there exists a pair of complementary nullcobordisms $(f_\pm\co C\to D_\pm,(\delta_\pm\phi,\phi))$ then in particular $0=H^2(C)=H^2(D_+)\oplus H^2(D_-)$ so that $H^2(D_\pm)=0$. But then $(T,\lambda)$ must be hyperbolic by Proposition \ref{hypcx}.

\end{proof}

\begin{remark}When $A$ is a Dedekind domain, we show in \cite[Theorem 4.8]{OrsonA} how to calculate the double Witt group $DW^\eps(A,A\sm\{0\})$ in terms of the Witt groups of forms over the residue class fields $A/pA$ where $pA$ is an involution invariant prime ideal, indeed this is the calculation underlying Example \ref{ex:calculation}.
\end{remark}

\begin{definition}A non-singular $\eps$--symmetric linking form $(T,\lambda)$ over $(A,S)$ is called \emph{stably hyperbolic} if there exist hyperbolic $\eps$--symmetric linking forms $H,H'$ such that $(T,\lambda)\oplus H\cong H'$.
\end{definition}

\begin{corollary}[`Stably hyperbolic $=$ hyperbolic']\label{cor:stablyhypishyp}Let $A$ be a ring with involution which contains a half-unit. A non-singular $\eps$--symmetric linking form $(T,\lambda)$ over $(A,S)$ is hyperbolic if and only if it is stably hyperbolic.
\end{corollary}

\begin{proof}`Only if' is clear. Conversely, a stably hyperbolic $\eps$--symmetric linking form determines the 0 class in $DW^\eps(A,S)\cong DL^0(A,S,\eps)$. Therefore there is a double-nullcobordism of the corresponding 1--dimensional $(-\eps)$--symmetric complex. But by the proof of Proposition \ref{iso}, these nullcobordisms correspond to complementary lagrangians.
\end{proof}

\begin{remark}For any $(A,S)$, Ranicki \cite[Proposition 3.4.7(ii)]{MR620795} proves an isomorphism $W^\eps(A,S)\cong L^0(A,S,\eps)$, but it is not sufficient to prove that stably metabolic implies metabolic for linking forms in general. The reason is that an $(A,S)$--nullcobordism $(f\co C\to D,(\delta\phi,\phi))$ of a 1--dimensional $(-\eps)$--symmetric $S$--acyclic Poincar\'{e} complex over $A$ might have $H^2(D)\neq0$, so that the corresponding $\eps$--symmetric linking form need not necessarily admit a lagrangian (cf.\ the torsion version of \cite[4.6]{MR560997}). Whereas in the hyperbolic case, we used the fact that $H^2(C)=0$ implies $H^2(D_\pm)=0$, exploiting the complementary condition present in our setup.
\end{remark}

\subsection{The linking and Blanchfield form of a symmetric Poincar\'{e} complex}

In topology, linking forms arise as the middle-dimensional torsion pairing on the homology (or cohomology) of a manifold. In this subsection we make clear, for a general chain complex with symmetric structure, when one should expect the middle-dimensional linking pairing to be a linking \emph{form}. The results in this subsection are required for our topological application to the Blanchfield forms of high-dimensional knot theory in Section \ref{sec:knots}.

Over a general $(A,S)$, the approach of simply taking the middle-dimensional torsion pairing of a symmetric Poincar\'{e} complex has two problems: the cohomology modules might not have homological dimension 1 (even when we restrict to the torsion), and the linking form might not pair modules to their torsion duals due to the universal coefficient problem. We now make clear some circumstances in which taking the middle-dimensional torsion pairing of a symmetric Poincar\'{e} complex $(C,\phi)$ is a valid operation from this perspective.

\begin{proposition}\label{prop:welldeflagrang}Suppose $(A,S)$ has the property that any $S$--torsion $A$--module has homological dimension 1 (this happens, for instance, if $A$ has homological dimension 1). Let $(C,\phi)$ be a $(2k+3)$--dimensional $\eps$--symmetric $S$--acyclic Poincar\'{e} complex over $A$. Then \[\lambda_\phi\co H^{k+2}(C)\times H^{k+2}(C)\to S^{-1}A/A;\qquad ([x],[y])\mapsto s^{-1}\overline{\tilde{y}(\phi_0(x))},\]with $x,y\in C^{k+2}$, $\tilde{y}\in C^{k+1}$, and $s\in S$ such that $d^*\tilde{y}=sy$, is a well-defined, non-singular, $(-1)^{k}\eps$--symmetric linking form. Moreover:\begin{enumerate}[(i)]
\item If $(C,\phi)$ is $(A,S)$--nullcobordant then $(H^{k+2}(C),\lambda_\phi)$ is metabolic.
\item If $(C,\phi)$ is $(A,S)$--double-nullcobordant then $(H^{k+2}(C),\lambda_\phi)$ is hyperbolic.
\end{enumerate}
\end{proposition}

\begin{proof}The first part is standard. The linking form is easily checked to be well-defined and the non-singularity comes from a standard universal coefficient spectral sequence argument. The $(-1)^{k}\eps$--symmetry can be derived from a chain-level calculation (which in general requires the use of the higher chain homotopy $\phi_1$), see for instance the chain-level calculations in Powell \cite[p.\ 151]{Powell:2011kq}.

For (i), suppose $(g\co C\to D,(\delta\phi,\phi))$ is an $(A,S)$--nullcobordism of $(C,\phi)$. Write the functor $e^1(-)=\Ext^1_A(-,A)$ for brevity. Then the long exact sequences of the morphism $g\co C\to D$ determine a commutative diagram with exact rows \[\xymatrix{H^{k+2}(D;A)\ar[r]^-{g^*}&H^{k+2}(C;A)\ar[r]&H^{k+3}(D,C;A))\\
e^1(H_{k+1}(D;A))\ar[r]\ar[u]_-{\cong}&e^1(H_{k+1}(C;A))\ar[r]\ar[u]_-{\cong}&e^1(H_{k+2}(D,C;A))\ar[u]_-{\cong}\\ 
e^1(H^{k+3}(D,C;A))\ar[u]^-{(\delta\phi_0\,\,\pm f\phi_0)}_-{\cong}\ar[r]&e^1(H^{k+2}(C;A))\ar[u]^-{\phi_0}_-{\cong}\ar[r]^-{e^1(g^*)}&e^1(H^{k+2}(D;A))\ar[u]^-{\lmat \delta\phi_0\\ \phi_0 f^*\rmat}_-{\cong}}\]The map adjoint to the linking form is the downwards composition of the central column. Hence the inclusion of the image $j\co g^*(H^{k+2}(D;A))\hookrightarrow H^{k+2}(C;A)$ is a lagrangian submodule as the commutative diagram determines an exact sequence\[0\to g^*(H^{k+2}(D;A))\xrightarrow{j} H^{k+2}(C;A)\xrightarrow{j^\wedge \lambda_\phi}g^*(H^{k+2}(D;A))^\wedge \to 0.\]

For (ii), suppose $(f_\pm\co C\to D_\pm,(\delta_\pm\phi,\phi))$ is an $(A,S)$--double-nullcobordism of $(C,\phi)$.  By the above we have that the direct sum decomposition $(f_+^*\,\,\,\,f_-^*)\co H^{k+2}(D_+)\oplus H^{k+2}(D_+)\cong H^{k+2}(C)$ is by lagrangians.
\end{proof}

\subsection*{The special algebraic case of Blanchfield forms}

The types of linking forms that arise in classical knot theory, called Blanchfield forms, are non-singular $\eps$--symmetric linking forms over $(R[z,z^{-1}],P)$ where $P$ is the set of \emph{Alexander polynomials} \[P:=\left\{p(z)\in R[z,z^{-1}]\,|\,p(1)\in R\text{ is a unit}\right\}.\]

The ring $R[z,z^{-1}]$ does not contain a half-unit necessarily. However, according to Ranicki \cite[Proposition 10.21(iv)]{MR1713074}, if we formally adjoin the half-unit $(1-z)^{-1}$ to ring, then there is an equivalence of exact categories\begin{equation}\label{eq:excision}\H(R[z,z^{-1}],P)\xrightarrow{\cong} \H(R[z,z^{-1},(1-z)^{-1}],P).\end{equation}Under this equivalence, an object $T$ of $\H(R[z,z^{-1}],P)$ corresponds to a homological dimension 1, f.g.\ $R[z,z^{-1}]$--module $T$ such that $1-z\co T\to T$ is an isomorphism. 

\begin{proposition}\label{prop:cartmorph}The equivalence of Equation \ref{eq:excision} induces an equivalence of categories of the corresponding non-singular $\eps$--symmetric linking forms. Under the equivalence of Equation \ref{eq:excision}, (split) lagrangians correspond to (split) lagrangians.

For each $n>1$, the equivalence of Equation \ref{eq:excision} induces an equivalence of categories of the corresponding $P$--acyclic $n$--dimensional $\eps$--symmetric (Poincar\'{e}) complexes, and of $P$--acyclic $(n+1)$--dimensional $\eps$--symmetric (Poincar\'{e}) pairs.
\end{proposition}

\begin{proof}The equivalence of Equation \ref{eq:excision} comes from a special case of a general Cartesian morphism of localisations of rings with involution (see \cite[p.\ 201]{MR620795}). The proof of Proposition \ref{prop:cartmorph} for general Cartesian morphisms can be found in Ranicki \cite[3.1.3, 3.6.2, 3.2.1]{MR620795}. See also \cite[\textsection 4]{MR2058802}.
\end{proof}

Now set $\Gamma=\Z[z,z^{-1}]$ and suppose $T$ is an $\Gamma$--module such that $\Hom_\Gamma(T,\Gamma)=0$. Set $t(T)$ to be the \emph{$\Z$--torsion} \[t(T)=\ker(T\to \Q[z,z^{-1}]\otimes_\Gamma T)\qquad \text{and} \qquad f(T)=T/t(T).\]Note that $f(T)$ may still have torsion with respect to the multiplicative subset $\Z[z,z^{-1}]\sm\{0\}$. Set $P$ to be the set of Alexander polynomials. Levine \cite{MR0461518} shows that for any module $T$ in $\H(\Gamma,P)$ the $\Z$--torsion and $\Z$--torsion free components are picked out as follows\[\begin{array}{rcl}\Ext_\Gamma^2(T,\Gamma))&\cong& t(T),\\
\Ext_\Gamma^1(T,\Gamma))&\cong& f(T).\\
\end{array}\]Now suppose $(C,\phi)$ is an $n$--dimensional $\eps$--symmetric $P$--acyclic Poincar\'{e} complex over $\Gamma$. Then by the universal coefficient spectral sequence collapse detailed in \cite{MR0461518} we obtain an isomorphism\[\begin{array}{rcl}f(H^r(C;\Gamma))&\xrightarrow{\cong}& H^r(C;\Gamma)/\Ext^2_\Gamma(H_{r-2}(C;\Gamma),\Gamma)\\&\xrightarrow{\cong}&\Ext_\Gamma^1(H_{r-1}(C;\Gamma,\Gamma))\\ &\xrightarrow{\cong}& \Ext_\Gamma^1(H^{(n+1)-r}(C;\Gamma,\Gamma))\cong \Hom_\Gamma(f(H^{(n+1)-r}(C;\Gamma)),P^{-1}\Gamma/\Gamma).\end{array}\]This isomorphism is adjoint to the following pairing:

\begin{definition}[{Levine \cite{MR0461518}}]\label{def:levblanch}Let $(C,\phi)$ be an $n$--dimensional $\eps$--symmetric $P$--acyclic Poincar\'{e} complex over $\Gamma$. Then the \emph{Blanchfield pairing} is the pairing \[Bl\co f(H^r(C;\Gamma))\times f(H^{(n+1)-r}(C;\Gamma))\to P^{-1}\Gamma/\Gamma;\qquad (x,y)\mapsto p^{-1}\overline{\tilde{y}(\phi(x))},\]where $x\in C^r$, $y\in C^{(n+1)-r}$, $\tilde{y}\in C^{n-r}$ and $p\in P$ such that $d^*\tilde{y}=py$.
\end{definition}

\begin{proposition}\label{prop:welldeflagrang2}Let $(C,\phi)$ be a $(2k+3)$--dimensional $\eps$--symmetric $P$--acyclic Poincar\'{e} complex over $\Gamma$. Then the \emph{Blanchfield form}:\[\lambda_\phi\co f(H^{k+2}(C))\times f(H^{k+2}(C))\to P^{-1}\Z[z,z^{-1}]/\Z[z,z^{-1}];\qquad ([x],[y])\mapsto p^{-1}\overline{\tilde{y}(\phi_0(x))},\]with $x,y\in C^{k+2}$, $\tilde{y}\in C^{k+1}$, and $p\in P$ such that $d^*\tilde{y}=sy$, is a well-defined, non-singular, $(-1)^{k}\eps$--symmetric linking form. Moreover: \begin{enumerate}[(i)]
\item If $(C,\phi)$ is $(\Gamma,S)$--nullcobordant then $(f(H^{k+2}(C)),\lambda_\phi)$ is metabolic.
\item If $(C,\phi)$ is $(\Gamma,S)$--double-nullcobordant then $(f(H^{k+2}(C)),\lambda_\phi)$ is hyperbolic.
\end{enumerate}
\end{proposition}

\begin{proof}It is shown in \cite{MR0461518} that the Blanchfield form is well-defined and non-singular. As the chain-level formula is identical to that of the linking form in \ref{prop:welldeflagrang}, the $(-1)^{k}\eps$--symmetry follows from the same calculations as in that proof.

For (i), suppose $(g\co C\to D,(\delta\phi,\phi))$ is an $(\Gamma,S)$--nullcobordism of $(C,\phi)$. We must appeal to results of Levine (see also Letsche \cite[\textsection 2.1]{MR1735303}). It is shown in \cite{MR0461518} that for any chain complex $C$ over $\Z[z,z^{-1}]$, the universal coefficient spectral sequence collapses to determine short exact sequences \[0\to \Ext^2_\Gamma(H_{r-2}(C;\Gamma),\Gamma)\to H^r(C;\Gamma)\to \Ext^1_\Gamma(H_{r-1}(C;\Gamma),\Gamma)\to 0.\] Write $e^1(-)=\Ext^1_\Gamma(-,\Gamma)$ for brevity. As the chain complexes $C$, $D$ and $C(f)$ are all $P$--acyclic, we obtain the following commutative diagram with exact rows\[\xymatrix{f(H^{k+2}(D;\Gamma))\ar[r]\ar[d]^-{\cong}&f(H^{k+2}(C;\Gamma))\ar[r]\ar[d]^-{\cong}&f(H^{k+3}(D,C;\Gamma))\ar[d]^-{\cong}\\
e^1(f(H_{k+1}(D;\Gamma)))\ar[r]&e^1(f(H_{k+1}(C;\Gamma)))\ar[r]&e^1(f(H_{k+2}(D,C;\Gamma)))\\ 
e^1(f(H^{k+1}(D,C;\Gamma)))\ar[u]^-{(\delta\phi_0\,\,\pm f\phi_0)}_-{\cong}\ar[r]&e^1(f(H^{k+2}(C;\Gamma)))\ar[u]^-{\phi_0}_-{\cong}\ar[r]&e^1(f(H^{k+2}(D;\Gamma)))\ar[u]^-{\lmat \delta\phi_0\\ \phi_0 f^*\rmat}_-{\cong}}\]As in Proposition \ref{prop:welldeflagrang}, the image of \[g^*\co f(H^{k+2}(D;\Gamma))\to f(H^{k+2}(C;\Gamma))\]is a lagrangian submodule.

For (ii), suppose $(f_\pm\co C\to D_\pm,(\delta_\pm\phi,\phi))$ is an $(\Gamma,S)$--double-nullcobordism of $(C,\phi)$.  By the above we have that $(f_+^*\,\,\,\,f_-^*)\co H^{k+2}(D_+)\oplus H^{k+2}(D_+)\cong H^{k+2}(C)$ is now a direct sum decomposition by complementary lagrangians.
\end{proof}

Propositions \ref{prop:welldeflagrang} and \ref{prop:welldeflagrang2} have the following corollary, which is well known, but worth stating in this very general form:

\begin{corollary}Suppose for a ring with involution $R$ and localisation $(R,S)$ we have either one of:\begin{enumerate}[(i)]\item Every $S$--torsion $R$--module has homological dimension 1.
\item $(R,S)=(\Z[z,z^{-1}],P)$.
\end{enumerate}
Then if $(T,\lambda)$ is a non-singular, $\eps$--symmetric linking form over $(R,S)$ that is stably metabolic it is moreover metabolic.
\end{corollary}

\begin{proof}Under the correspondence of Proposition \ref{prop:correspondence}, $(T,\lambda)$ goes to a 1--dimensional $(-\eps)$--symmetric $S$--acyclic Poincar\'{e} complex $(C,\phi)$ over $A$. If $(T,\lambda)$ is stably metabolic, there exists an $(A,S)$--nullcobordism $(f\co C\to D,(\delta\phi,\phi))$ by the isomorphism $W^\eps(A,S)\cong L^2(A,S,-\eps)$ (\cite[Proposition 3.4.7(ii)]{MR620795}). But then by Proposition \ref{prop:welldeflagrang} (in the case of (i)) or by Proposition \ref{prop:welldeflagrang2} (in the case of (ii)), $(H^1(C),\lambda_\phi)=(T,\lambda)$ is metabolic.
\end{proof}

And the following corollary is clear from Propositions \ref{prop:welldeflagrang} and \ref{prop:welldeflagrang2}.

\begin{corollary}Suppose for a ring with involution $R$ and localisation $(R,S)$ we have either one of:
\begin{enumerate}[(i)]
\item Every $S$--torsion $R$--module has homological dimension 1 and $R$ contains a half-unit.
\item $(R,S)=(\Z[z,z^{-1},(1-z)^{-1}],P)$.
\end{enumerate}

Then Proposition \ref{prop:correspondence} defines a surjective homomorphism\[DL^{2k+2}(R,S,(-1)^{k+1}\eps)\twoheadrightarrow DW^\eps(R,S),\]with right inverse given by the isomorphism $DW^\eps(R,S)\cong DL^0(R,S,\eps)$ followed by the $(k+1)$--fold skew-suspension $\overline{S}^{k+1}$.
\end{corollary}

\section{Double $L$--groups obstruct double knot-cobordism}\label{sec:knots}

In Section \ref{sec:knots} we apply the new algebraic results of this paper to the setting of high-dimensional knot theory. We will prove several new results relating to doubly slice knots and reprove some known results using our techniques.

More specifically, for an $n$--dimensional knot, we will recall how to define a knot invariant which uses the entire chain complex of a knot exterior, called the \emph{Blanchfield complex}. We are expanding the details of a construction originally made by Ranicki \cite[\textsection 7.9]{MR620795}. We will prove that the class of the Blanchfield complex of a doubly slice knot vanishes in $DL^{n+1}(\Z[z,z^{-1},(1-z)^{-1}],P)$. For $n$ odd, this result will then be related to the Seifert and Blanchfield forms of the knot. Using the algebraic results in the earlier sections of this paper we show that for Blanchfield forms, Seifert forms and Blanchfield complexes that `algebraically stably doubly slice implies doubly slice'.

We note that while the original definition of the Blanchfield complex provided an elegant formulation for the slice problem, the use of the full chain complex of the knot exterior was unnecessary for Kervaire \cite{MR0189052} and Levine's \cite{MR0246314} solution to this problem. However, its use as an approach to the doubly slice problem is motivated by the results of Ruberman \cite{MR709569,MR933307}, where it is shown that there exist high-dimensional doubly slice invariants beyond the middle-dimensional pairings used by Kervaire and Levine. This suggests the use of something like the Blanchfield complex really is necessary to study doubly slice knots. On the other hand Ruberman's work also suggests that the fundamental group $\pi_1(S^{n+2}\sm K)$ plays a vital role in this problem, even high-dimensionally, so the Blanchfield complex over $\Z[\Z]$ which we will use below cannot be the full story. These concerns are discussed in our closing remarks.

\begin{notation}For the rest of the paper, we use the notation $\Lambda=\Z[z,z^{-1},(1-z)^{-1}]$.\end{notation}

\subsection{Basic high-dimensional knot theory}

A \emph{topological $n$--knot}, also called a \emph{knot} unless $n$ is to be specified, is an ambient isotopy class of oriented, locally flat embeddings $K\co S^n\hookrightarrow S^{n+2}$ (where all spheres are considered to have a preferred orientation already). In a standard abuse of notation we will also use the word knot to mean a particular $K$ in an ambient isotopy class and the image of $K$ in $S^{n+2}$. The \emph{unknot} is the ambient isotopy class of $U\co S^n\hookrightarrow S^{n+2}$, the standard unknotted $n$--sphere in the unit sphere $S^{n+2}\subset\R^{n+3}$ given by setting the last two co-ordinates to 0. The \emph{inverse knot} $-K$ of a knot $K$ is given by reversing the orientation on a mirror image of $K$ in $S^{n+2}$. Any embedding $K\co S^{n}\hookrightarrow S^{n+2}$ has trivial normal bundle and hence, by choosing a framing, we may excise a small, trivial tubular neighbourhood of the knot from $S^{n+2}$. Thus, the \emph{knot exterior} is the manifold with boundary \[(X_K,\partial X_K):=(\closure (S^{n+2}\sm (K(S^n)\times D^2)),S^n\times S^1)\]which has a preferred orientation coming from the ambient $S^{n+2}$. The knot exterior $X_K$ is homotopy equivalent to the \emph{knot complement} $S^{n+2}\sm K$ and hence has the homology of a circle $H_*(X_K)= H_*(S^1)$ by Alexander duality.

If there is a locally flat embedding of the manifold with boundary $(F^{n+1},S^n)\hookrightarrow S^{n+2}$ then we say the embedded $F$ is a \emph{Seifert surface} for the boundary knot. Every knot $K$ admits a Seifert surface $F^{n+1}$ (see Kervaire \cite{MR0189052} or Zeeman \cite{MR0160218}). For $n\neq 2$, the unknot is characterised as the only knot which admits $D^{n+1}$ as a Seifert surface.

It is always possible to `push' a Seifert surface into the standard $D^{n+3}$ that cobounds the ambient sphere $S^{n+2}$. That is, we may modify a locally flat embedding $(F^{n+1},S^{n})\hookrightarrow S^{n+2}$ to a locally flat embedding of pairs $(F^{n+1},S^{n})\hookrightarrow (D^{n+3},S^{n+2})$, without changing the ambient isotopy class of the bounding knot $K$, and so that the embedded $F$ intersects $S^{n+2}$ in the knot $K$. If there is a locally flat embedding of pairs $(D,K)\co (D^{n+1},S^n)\hookrightarrow (D^{n+3},S^{n+2})$ then we say the knot $K$ is \emph{slice} and the locally flat embedding $D$ is a \emph{slice disc} for $K$.

Any codimension 2 submanifold pair $(N,\partial N)\subset (D^{n+3},S^{n+2})$ has trivial normal bundle (see for instance Ranicki \cite[Proposition 22.1]{MR1713074}), and hence by choosing a  framing we may embed $(N,\partial N)\times D^2\subset (D^{n+3},S^{n+2})$. Define the \emph{exterior} of such a submanifold pair (with respect to a choice of framing) as the compact, oriented manifold triad\[(Y_N;X_{\partial N},\partial_+Y_N;\partial X_{\partial N}):=(\closure (D^{n+3}\sm (N\times D^2));\closure (S^{n+2}\sm (\partial N\times D^2)),N\times S^1;\partial N\times S^1).\]

\begin{figure}[h]\[\def\picextYone{\resizebox{0.4\textwidth}{!}{ \includegraphics{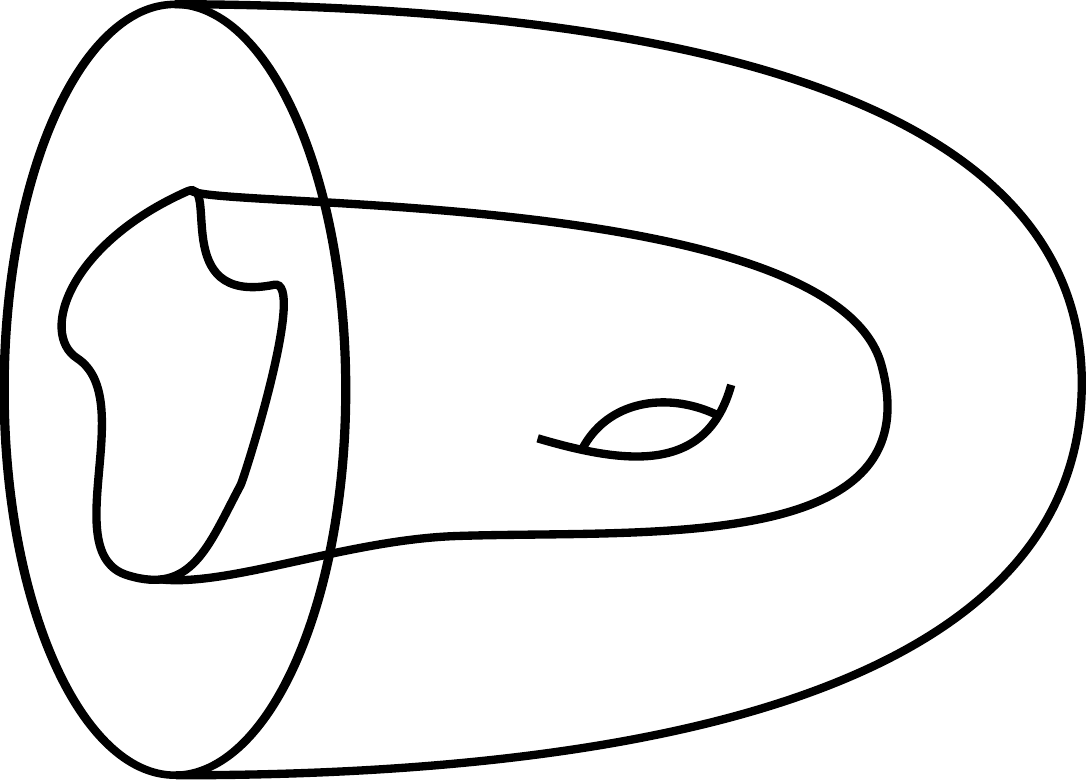}}}
\begin{xy} \xyimport(300,300){\picextYone}
,!+<7pc,-2.2pc>*+!\txt{Schematic of the framed \\codimension 2 submanifold pair.}
,(35,330)*!L{S^{n+2}}
,(190,320)*!L{D^{n+3}}
,(33,255)*!L{X_{\partial N}}
,(-70,240)*!L{\partial N\times S^1}
,(-40,225)*+{}="A";(30,200)*+{}="B"
,{"A"\ar"B"}
,(-70,70)*!L{\partial N\times D^2}
,(-30,80)*+{}="C";(55,110)*+{}="D"
,{"C"\ar"D"}
,(160,250)*!L{Y_N}
,(110,180)*!L{N\times D^2}
,(120,40)*!L{N\times S^1}
,(145,47)*+{}="A";(150,97)*+{}="B"
,{"A"\ar"B"}
\end{xy}
\qquad\quad
\def\picrelboundary{\resizebox{0.4\textwidth}{!}{ \includegraphics{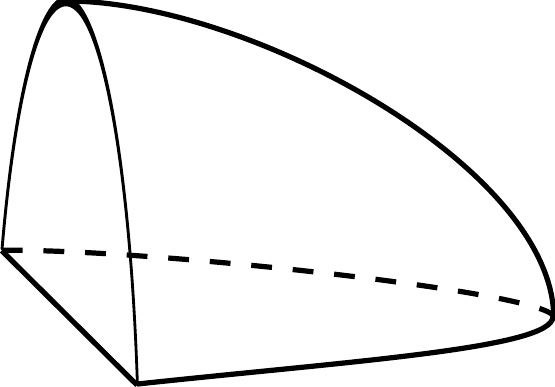}}}
\begin{xy} \xyimport(300,300){\picrelboundary}
,!CD+<0.3pc,-1pc>*+!CU\txt{Schematic of the exterior \\as a triad.}
,(23,180)*!L{X_{\partial N}}
,(0,30)*!L{\partial X_{\partial N}}
,(140,170)*!L{Y_N}
,(120,50)*!L{\partial_+Y_N}
\end{xy}\]
\end{figure}

\subsection{The Blanchfield complex, knot-cobordism and $L$--theory}\label{subsec:blanchfield}

The \emph{Blanchfield complex} of a knot will be our central object of study and is the bridge between the algebraic $L$--theory of the previous sections and our knot theoretic applications. The Blanchfield complex is an invariant of an ambient isotopy class of embeddings $K\co S^n\hookrightarrow S^{n+2}$ that is defined for both odd- and even-dimensional knots. It is the symmetric chain complex generalisation of the classical knot invariant called the \emph{Blanchfield form}, which is defined only for odd-dimensional knots. We will define the Blanchfield form of an odd-dimensional knot below and show how it derives from the Blanchfield complex.

First we spell out the details of the construction of the Blanchfield complex of an $n$--knot $K$, originally defined by Ranicki in \cite[p.\ 822]{MR620795}. We will need the following well-known proposition whose proof is standard obstruction theory.

\begin{proposition}Suppose $f\co (F^{n+1},S^n)\hookrightarrow (D^{n+3},S^{n+2})$ is a locally flat embedding of pairs, and write $f|_{S^n}=K$. Then there is a \emph{meridian map}. That is, a map\[m\co Y_F\to S^1,\]inducing an isomorphism $m_*\co H_*(X_K)\cong H_*(S^1)$ and restricting to projection to the second factor\[m|_{ \partial_+ Y}=\text{pr}_2\co  F\times S^1\to S^1.\]The meridian map is uniquely defined up to homotopy by the fact that it restricts to projection on $\partial_+Y$. If $F=D$ is a slice disc then the meridian map induces an isomorphism $m_*\co H_*(Y_D)\cong H_*(S^1)$.
\end{proposition}

\begin{corollary}\label{cor:meridian}There is a \emph{meridian map} on the knot exterior\[m\co X_K\to S^1\]uniquely defined up to homotopy by the property that $m|_{\partial X_K}\co S^n\times S^1\to S^1$ is projection to the second factor.
\end{corollary}

The homotopy class of the meridian map $m\in[X_K,S^1]=[X_K,D^{n+1}\times S^1]$ may be represented by a (degree 1) map of compact, oriented, $(n+2)$--dimensional manifolds with boundary\[(f,\partial f)\co (X_K,\partial X_K)\to (D^{n+1}\times S^1,S^n\times S^1),\]with $\partial f$ the identity map.

Using the standard infinite cyclic cover $D^{n+1}\times \R\to D^{n+1}\times S^1$ with group of covering translations $\Z\cong\langle z\rangle$, we may now apply Ranicki's \emph{symmetric construction} \cite[Proposition 6.5]{MR566491} to obtain the associated kernel pair $\sigma^*(\overline{f},\overline{\partial f})$, which is an $(n+2)$--dimensional symmetric Poincar\'{e} pair over the Laurent polynomial ring $\Z[\Z]\cong\Z[z,z^{-1}]$ with the involution $\overline{z}= z^{-1}$. The underlying morphism of the pair $\sigma^*(\overline{f},\overline{\partial f})$ is given by the morphism $g$ of mapping cones induced by the following diagram \[\xymatrix{C(\overline{f}^!)&C_*(\overline{X_K})\ar[l]\ar[r]_-{\overline{f}_*}&C_*(D^{n+1}\times \R)\ar@/_1pc/[l]_-{\overline{f}^!}\\C(\overline{\partial f}^!)\ar[u]^-{g}&C_*(\overline{\partial X_K})\ar[l]\ar[u]\ar[r]_-{(\overline{\partial f})_*}&C_*(S^n\times\R)\ar[u]\ar@/_1pc/[l]_-{\overline{\partial f}^!}}\]where the \emph{chain level Umkehr maps} $\overline{f}^!$ and $\overline{\partial f}^!$ are defined using Poincar\'{e} and Poincar\'{e}--Lefschetz duality. We refer the reader to \cite[Proposition 6.5]{MR566491} for full details.

\begin{definition}The \emph{Blanchfield complex} of an $n$--knot $K\co S^n\hookrightarrow S^{n+2}$ is the $(n+2)$--dimensional symmetric complex $(C_K,\phi_K)$ over $\Z[z,z^{-1}]$ defined as the algebraic Thom construction of the kernel pair $\sigma^*(\overline{f},\overline{\partial f})$.
\end{definition}

Identifying $(D^{n+1}\times S^1,S^n\times S^1)\cong (X_U,\partial X_U)$, the Blanchfield complex of $K$ can be thought of as a measure of the difference between $K$ and the unknot $U$. In other words, we can think of the Blanchfield complex as a surgery problem trying to improve the knot exterior to an unknot exterior via codimension 2 surgery (see Ranicki \cite[\textsection 7.8]{MR620795}). The Blanchfield complex is an invariant of the ambient isotopy class of $K$ that is well-defined up to homotopy equivalence of $(n+2)$--dimensional symmetric complexes over $\Z[z,z^{-1}]$.

\begin{proposition}\label{prop:auto}The Blanchfield complex $(C_K,\phi_K)$ of an $n$--knot $K$ is Poincar\'{e} and such that\[C_K\oplus C_*(\overline{D^{n+1}\times S^1})\simeq C_*(\overline{X_K}).\]So in particular there is an isomorphism in reduced homology $\widetilde{H}_*(C_K)\cong \widetilde{H}_*(X_K)$. Furthermore,\[1-z\co C_K\to C_K\]is an automorphism of $C_K$.
\end{proposition}

\begin{proof}By \cite[Proposition 1.3.3]{MR620795}, a symmetric complex is Poincar\'{e} if and only if it is the Thom construction of a pair that is homotopy equivalent to a pair of the form\[(0\co 0\to D,(\phi,0)).\]But indeed,  $\partial f=\text{id}\co \partial X_K\to \partial X_U$ implies that the chain level Umkehr map $\overline{\partial f}^!$ in $h\B_+(\Z[z,z^{-1}])$ is a chain homotopy equivalence, so that $C(\overline{\partial f}^!)$ is contractible and the kernel pair $\sigma^*(\overline{f},\overline{\partial f})$ is of the required form. The direct sum decomposition of \cite[Proposition 6.5]{MR566491} reduces to the claimed decomposition under the algebraic Thom construction.

The augmentation $\eps\co \Z[z,z^{-1}]\to \Z$ sending $z\mapsto 1$ fits into the free $\Z[z,z^{-1}]$--module resolution\[0\to \Z[z,z^{-1}]\xrightarrow{1-z}\Z[z,z^{-1}]\xrightarrow{\eps} \Z\to 0.\]Applying this coefficient sequence to $C_K$ shows that the statement that $1-z$ acts as an automorphism of $C_K$ is equivalent to saying that $\Z\otimes_{\Z[z,z^{-1}]} C_K$ is acyclic. But it is easy to see that $\Z\otimes_{\Z[z,z^{-1}]} C_K$ is acyclic as the original map $(f,\partial f)$ was a $\Z$--homology equivalence (by Alexander duality, as already noted). Taking the cone $C(\overline{f}^!)$ and forgetting the action of the covering translations results in an acyclic complex.
\end{proof}

Recall that $P$ denotes the set of Alexander polynomials.

\begin{lemma}\label{lem:auto}If $C$ is a chain complex in $h\B_+(\Z[z,z^{-1}])$ then $(1-z)\co C\to C$ is an automorphism if and only if there exists $p\in P$ such that $pH_*(C)=0$.
\end{lemma}

\begin{proof}
Suppose $H_*(C)$ is $P$--torsion. Then, as localisation is exact, $H_*(P^{-1}\Z[z,z^{-1}]\otimes_{\Z[z,z^{-1}]}C)=0$. The augmentation map $\eps\co \Z[z,z^{-1}]\to \Z$ from above factors as\[\eps\co \Z[z,z^{-1}]\to P^{-1}\Z[z,z^{-1}]\to \Z\]because $p(1)\in\Z$ is a unit for all $p\in P$. Hence $H_*(\Z\otimes_{\Z[z,z^{-1}]}C)=0$, which has already been observed to be equivalent to saying that $(1-z)\co C\to C$ is an automorphism.

The converse follows just as in the proof of Levine \cite[Corollary 1.3]{MR0461518}.
\end{proof}

\begin{corollary}\label{cor:Pacyclic}The Blanchfield complex $(C_K,\phi_K)$ of an $n$--knot $K$ is an $(n+2)$--dimensional $P$--acyclic symmetric Poincar\'{e} complex over $\Z[z,z^{-1}]$. Equivalently, by Proposition \ref{prop:cartmorph}, $(C_K,\phi_K)$ is an $(n+2)$--dimensional $P$--acyclic symmetric Poincar\'{e} complex over $\Lambda$.
\end{corollary}

The set of ambient isotopy classes of $n$--knots, equipped with the operation of connected sum of knots $K_1\# K_2\co S^n\hookrightarrow K_1(S^n)\# K_2(S^n)$ is a commutative monoid called $Knots_n$, with unit given by the unknot $U$. The \emph{$n$--dimensional (topological) knot-cobordism group} is the group given by the monoid construction\[\mathcal{C}_n:=Knots_n/\{\text{slice knots}\}.\]The assignment to a knot of its Blanchfield complex \[\sigma^L\co \mathcal{C}_n\to L^{n+1}(\Lambda,P,-1);\qquad K\mapsto (C_K,\phi_K)\]is shown by Ranicki \cite[\textsection 7.8]{MR620795} to give a well-defined group homomorphism. We refer the reader to the author's thesis \cite[Lemmas 6.3.7, 6.3.8]{Orsonthesis} for a proof using only the tools developed in this paper.

\subsection{Seifert and Blanchfield forms of a $(2k+1)$--knot}

In this subsection we briefly recall some standard definitions, and a theorem from another paper by the author, which we shall need for the next subsection. 

Suppose $n=2k+1$. An $n$--dimensional knot has two very tractable and well-understood homological invariants, called the Blanchfield and Seifert forms of the knot. The Blanchfield form can be defined directly from the Blanchfield complex but the Seifert form depends on a choice of Seifert surface $j\co F\hookrightarrow S^{n+2}$ for the knot $K$ so cannot be derived from the Blanchfield complex without this extra information. However, the (double) Witt class of the Blanchfield form \emph{does} determine the (double) Witt class of any choice of Seifert form.

\begin{definition}The \emph{Blanchfield form for $K$} is the non-singular $(-1)^{k}$--symmetric linking form over $(\Z[z,z^{-1}],P)$ defined by Proposition \ref{prop:welldeflagrang2}\[Bl\co f(H^{k+2}(C_K))\times f(H^{k+2}(C_K))\to P^{-1}\Z[z,z^{-1}]/\Z[z,z^{-1}].\]
\end{definition}

\begin{remark}This definition is Poincar\'{e} dual to the common definition of a Blanchfield form as given for example by Levine \cite{MR0461518}.
\end{remark}

If $P$ is a f.g.\ $\Z$--module, denote the torsion-free component by $f(P):=P/TP$.

\begin{definition}\label{def:seifert}Given a choice of Seifert surface, the \emph{Seifert form of $(F,K)$} is the (well-defined) $(-1)^{k+1}$--symmetric Seifert form $(f(H^{k+1}(F)),\psi)$ over $\Z$ given by\[\psi\co f(H^{k+1}(F))\times f(H^{k+1}(F))\to \Z;\qquad(u,v)\mapsto l(x,i_*(y)),\]where $l$ denotes linking number, $x=u\cap[F]$, $y=v\cap[F]$ and $i:F\to S^{n+2}\sm F$ is defined by translation in the positive normal direction. It has the property that $(f(H^{k+1}(F)),\psi+(-1)^{k+1}\psi^*)$ is the non-singular, $(-1)^{k+1}$--symmetric middle-dimensional cohomology intersection pairing of $F$.
\end{definition}

Any choice of Seifert form determines the Blanchfield form via the algebraic analogue of the cut-and-paste construction of the infinite cyclic cover of the knot exterior. For precise details of this, see Levine \cite{MR0461518}. On the level of Witt groups, we show in another paper \cite{OrsonA} that the \emph{algebraic} cut-and-paste construction gives rise to the following:

\begin{theorem}[Covering isomorphism {\cite[4.8]{OrsonA}}]\label{thm:covering}For any $R$, there is an isomorphism of groups\[B:\widehat{DW}_\eps(R)\xrightarrow{\cong} DW^{-\eps}(R[z,z,(1-z)^{-1}],P),\]where $B$ is the algebraic covering morphism of Ranicki \cite{MR2058802}.
\end{theorem}

\subsection{Double knot-cobordism and double $L$--theory}

Recall the definition of a \emph{doubly slice} knot from Section \ref{sec:intro}. We finally show that the Blanchfield complex gives the desired invariant of doubly slice knots. As a consequence of this and the various algebraic work we have done in Sections \ref{sec:DLgroups} and \ref{sec:DWgroups} we obtain several new results for high-dimensional doubly slice knots, and some reproofs (from a very different perspective) of previously known results.

If $K, K'$ are doubly slice $n$--knots then Stoltzfus \cite{MR521738} showed $K\#K'$ is also doubly slice. Hence the doubly slice knots form a closed submonoid of $Knots_n$.

\begin{definition}The \emph{$n$--dimensional (topological) double knot-cobordism group} is the group given by the monoid construction\[\mathcal{DC}_n:=Knots_n/\{\text{doubly slice knots}\}.\]Simple doubly slice knots are similarly a closed submonoid of $Knots_n^\text{simp}$. So define\[\mathcal{DC}_n^\text{simp}:=Knots_n^\text{simp}/\{\text{simple doubly slice knots}\}.\]
\end{definition}

\begin{proposition}\label{prop:invariant}If $K$ is doubly slice then the Blanchfield complex $(C_K,\phi_K)$ is algebraically double-nullcobordant.
\end{proposition}
\begin{proof}We must check that Blanchfield complex of a doubly slice knot $K$ admits complementary $(\Lambda,P)$--nullcobordisms. Given a single slice disc $(D,K)$, represent the homotopy class of the meridian map $m\in[Y_D,S^1]=[Y,D^{n+3}\times S^1]$ by a (degree 1) map of compact oriented manifold triads\[F=(f,\partial f, \partial'f,\partial\partial f)\co (Y_D;X_K,\partial_+ Y_D;\partial X_K)\to(D^{n+3}\times S^1; D^{n+1}\times S^1,D^{n+1}\times S^1;S^n\times S^1),\]where both $\partial'f$ and $\partial\partial f$ are identity maps. (We can think of $F$ as a map from the slice disc exterior to the `trivial slice disc' exterior.)

\begin{figure}[h]\[\def\picslicetriad{\resizebox{0.85\textwidth}{!}{ \includegraphics{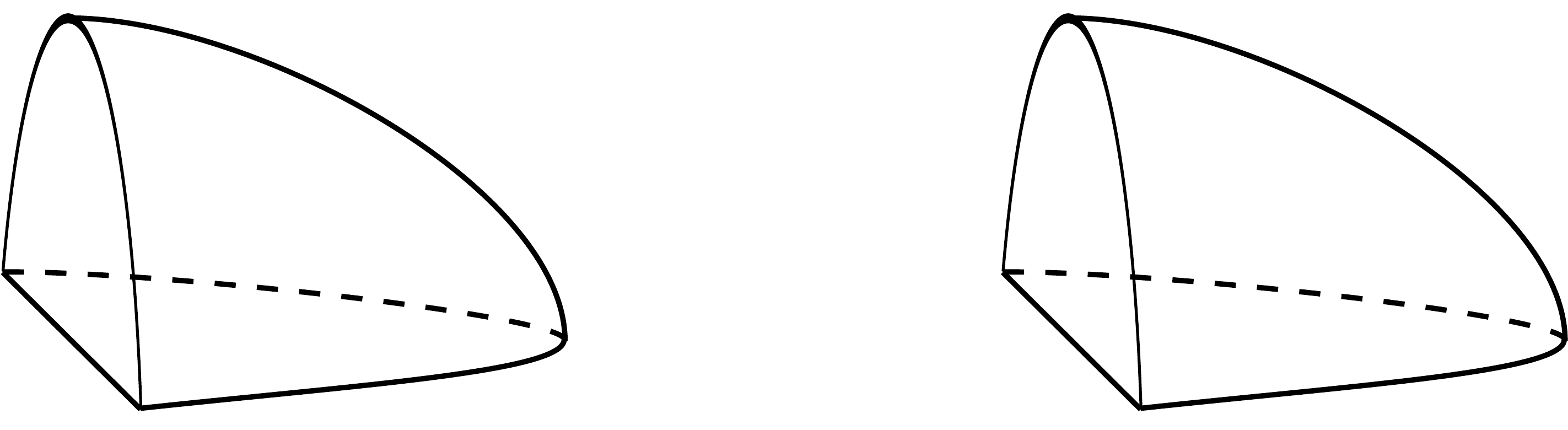}}}
\begin{xy} \xyimport(259,139){\picslicetriad}
,!+<7.2pc,2pc>*+!\txt{}
,(6,80)*!L{X_K}
,(-5,20)*!L{\partial X_K}
,(30,30)*!L{\partial_+ Y_D}
,(40,80)*!L{Y_D}
,(129,70)*!L{D^{n+1}\times S^1}
,(145,25)*!L{S^{n}\times S^1}
,(192,30)*!L{D^{n+1}\times S^1}
,(195,80)*!L{D^{n+2}\times S^1}
,(85,100)*+{}="A";(160,100)*+{}="B"
,{"A"\ar@/^/"B"}
,(120,123)*!L{F}
\end{xy}\]
  \caption{A schematic of the surgery problem determining the kernel triad.}
\end{figure}
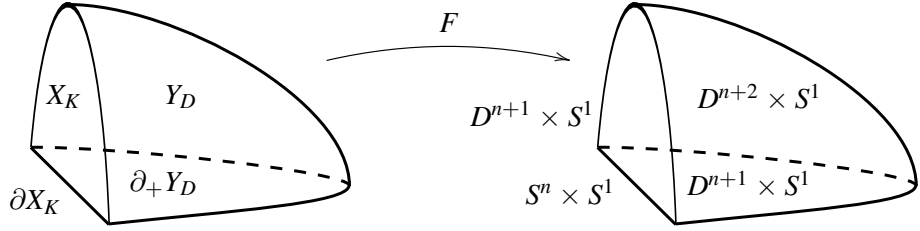

Using the standard $\Z$--cover $D^{n+3}\times\R\to D^{n+3}\times S^1$ we use Ranicki's symmetric construction to obtain the kernel triad $\sigma^*(\overline{F}^!)=(\Gamma,(\Phi,\delta\phi,\delta'\phi,\phi))$, some $(n+3)$--dimensional symmetric Poincar\'{e} triad in $h\B_+(\Z[z,z^{-1}])$. Now the relative algebraic Thom construction on this triad results in the set of pairs $\{x,x'\}$. But as $C(\partial\partial f^!),C(\partial' f^!)\simeq 0$, we have\[x=(C(\overline{\partial f}^!)\to C(\overline{f}^!),(\Phi/0,\delta\phi/0)),\qquad x'=(0\to C(C(\overline{\partial f}^!)\to C(\overline{f}^!)),(\Phi/\delta\phi,0/0)).\]Hence we have the pair $x=(C_K\to C(\overline{f}^!),(\Phi,\phi_K))$ for $(C_K,\phi_K)$ the Blanchfield complex of $K$. We need to show $x$ is an algebraic nullcobordism in $h\C_+(\Z[z,z^{-1}],P)$, in other words that $x$ is a $P$-acyclic Poincar\'{e} pair. But the kernel triad $\sigma^*(\overline{F}^!)$ is Poincar\'{e}, so by definition of a Poincar\'{e} triad $(0\cup_0 C_K\to C(\overline{f}^!),(\Phi,0\cup_0\phi_K))=x$ is Poincar\'{e}. Furthermore,  the meridian map $m$ is a homology equivalence on $Y_D$ and on $X_K$, so by the same arguments as in Proposition \ref{prop:auto} and Lemma \ref{lem:auto}, the Poincar\'{e} pair $x$ is moreover in $h\C_+(\Z[z,z^{-1}],P)$. Now, by Proposition \ref{prop:cartmorph}, $x$ is also a $(\Lambda,P)$-- nullcobordism

Hence taking now a \emph{pair} of complementary slice discs $(D_\pm, K)$ results in a pair of morphisms of compact oriented manifold triads\[F_\pm\simeq (f_\pm;\partial f_\pm,\text{id};\text{id})\co (Y_{D_\pm};\partial X_K,\partial_0Y_{D_\pm};\partial X_{\partial K})\to (Y_{U};\partial X_U,\partial_0Y_{U};\partial X_{U}),\]which result in a pair of $(\Lambda,P)$--nullcobordisms of the Blanchfield complex \[x_\pm\simeq(C_K\to C(\overline{f}_\pm^!),(\Phi_\pm,\phi_K)).\]We wish to check that the algebraic union $x_+\cup x_-\simeq 0$. But the underlying chain complex of an algebraic glueing is given by a certain mapping cone on the chain level. The same is true for the algebraic Thom construction on a pair, and the construction of the kernel triads. We may perform these mapping cones in any order and receive the same result, hence the underlying chain complex of $x_+\cup x_-$ is the result of performing these operations in the following order: glue the maps of algebraic triads $\overline{F_+}^!\cup \overline{F_-}^!$ (by glueing the triads along the knot exteriors $C(X_K)$), form the kernel triad $\sigma(\overline{F_+}^!\cup \overline{F_-}^!)$, then perform the algebraic Thom construction on this triad. But as the slice discs $(D,K)$ were complementary, we have that the union $D_+\cup_K D_-$ is unknotted in $S^{n+3}$ and hence $\overline{F_+}^!\cup \overline{F_-}^!$ is chain homotopic to the identity. Therefore the triad $\sigma(\overline{F_+}^!\cup \overline{F_-}^!)$ is contractible and hence $x_+\cup x_-\simeq 0$ as required.
\end{proof}

\begin{remark}The transitivity of the double $L$--groups mean that we have just given a partial affirmative answer to an algebraic question of Levine \cite[3(2)]{MR718271}. There is no cup-product type algebra structure in algebraic $L$--theory, so we have not yet completely answered Levine's question. However we conjecture that the techniques of double $L$--theory could be modified to include this sort of product structure and answer this question affirmatively.
\end{remark}

The following corollary was already shown to be true by Kearton \cite[Corollary 3]{MR0385873}, but Proposition \ref{prop:invariant}, combined with Proposition \ref{prop:welldeflagrang2} gives a different proof.

\begin{corollary}\label{cor:blanchfield}The Blanchfield form of an odd-dimensional doubly slice knot is hyperbolic.
\end{corollary}

We now have the following group homomorphisms obstructing double knot-cobordism:

\begin{corollary}For $n\geq 1$, there is a well-defined homomorphism\[\sigma^{DL}\co \mathcal{DC}_n\to DL^{n+1}(\Lambda,P,-1);\qquad [K]\mapsto (C_K,\phi_K).\]When $n=2k+1$ there is a well-defined homomorphism\[\sigma^{DW}\co \mathcal{DC}_n\to DW^{(-1)^k}(\Lambda,P);\qquad [K]\mapsto (f(H^{k+2}(C)),\lambda_{\phi_K}),\]and for any choice of Seifert surface $F$ there is a well-defined homomorphism\[\sigma^{\widehat{DW}}\co \mathcal{DC}_n\to \widehat{DW}_{(-1)^{k+1}}(\Z);\qquad [K]\mapsto (f(H^{k+1}(F)),\psi).\](This final morphism uses Theorem \ref{thm:covering}.)
\end{corollary}

We also combine some of the algebraic results from Section \ref{sec:DWgroups} to prove a new result about Seifert forms for knots.

\begin{theorem}\label{thm:theorem}Every Seifert form for an odd-dimensional doubly slice knot $K$ is hyperbolic.
\end{theorem}

\begin{proof}Combining Theorem \ref{thm:covering} and Corollary \ref{cor:blanchfield} shows that every Seifert form for $K$ is \emph{stably} hyperbolic. But by Corollary \ref{cor:stablyhypishypseif}, stably hyperbolic Seifert forms are moreover hyperbolic.
\end{proof}

We now have several algebraic responses to Question \ref{q:stablyhypishyp}:

\begin{theorem}\label{thm:stablyresults}Suppose for $n\geq1$ that an $n$--knot $K$ is stably double slice. Then the class $\sigma^{DL}(K)\in DL^{n+1}(\Lambda,P)$ of the Blanchfield complex vanishes. If $n=2k+1$ then the Witt classes of the Blanchfield form $\sigma^{DW}(K)\in DW^{(-1)^k}(\Lambda,P)$ and any choice of Seifert form $\sigma_{\widehat{DW}}(K)\in \widehat{DW}_{(-1)^{k+1}}(\Z)$ vanish.
\end{theorem}

If we assume we are dealing with a simple knot, the algebraic results of Section \ref{sec:DWgroups} yield the following partial answer to Question \ref{q:stablyhypishyp}.

\begin{theorem}\label{thm:stoltzfus}For odd $n=2k+1>1$, a simple $n$--knot $K$ has $[K]=0\in \mathcal{DC}^{simp}_n$ if and only if $K$ is doubly slice.
\end{theorem}

\begin{proof}`If' is clear. Conversely, if $\sigma^{DL}(K)=0$, we have that the Blanchfield form $(T,\lambda)$ for $K$ has $(T,\lambda)=0\in DW^{(-1)^k}(\Lambda,P)$. But by Corollary \ref{cor:stablyhypishyp}, this means $(T,\lambda)$ is hyperbolic. Hence any Seifert surface $F$ for $K$ has hyperbolic Seifert form by Corollary \ref{cor:stablyhypishypseif}. Take a basis of $H^{k+1}(F;\Z)$ with respect to which the matrix of the Seifert form is hyperbolic. The Poincar\'{e} dual basis to this can be realised by framed, embedded $(k+1)$--spheres which can be used as instructions for surgery on $F$ to realise two complementary slice discs as in \cite[Theorem 3.1]{MR0290351} (case $k>1$) and \cite{MR0380817} (case $k=1$).
\end{proof}

This is not the first proof of Theorem \ref{thm:stoltzfus}. In \cite{MR833015}, Bayer--Fl{\"u}ckiger and Stoltzfus obtain a slightly less general form of Corollary \ref{cor:stablyhypishyp} by very different methods to our own. The authors derive Theorem \ref{thm:stoltzfus} from this.

We finish with some remarks highlighting some subtleties of the doubly slice problem, contrasting the slice problem, and indicating possible directions for future investigations.
\begin{remark}$\,$

\begin{flushenumerate}
\item According to the work of Ruberman \cite[Theorem 4.17]{MR709569}, \cite[Theorem 3.3]{MR933307}, in every odd dimension, there exists an infinite family of knots with hyperbolic Blanchfield form but which are not doubly slice. When $n\neq 1$ all knots in the family have exteriors which are homotopy equivalent (rel.\ boundary, preserving meridians) to one another and to a doubly slice knot. When $n=1$ the exteriors have the same $\Z[\Z]$--homology type. There is a similar result in all even-dimensions \cite{MR709569}, \cite{MR933307}. One consequence is that there can be no general procedure that modifies a knot within its double knot-cobordism class to be simple. \emph{There is no double surgery below the middle dimension.} That is\[\mathcal{DC}_n\not\cong\mathcal{DC}^{simp}_n.\]

\item The mechanism for detecting non-doubly slice knots in Ruberman's work is a high-dimensional application of the Casson--Gordon invariants. The definition and non-vanishing of these invariants requires interesting cyclic representations of the fundamental group $\pi_1(S^{n+2}\sm K)$. There are no known Ruberman-type examples for $\pi_1(S^{n+2}\sm K)\cong\Z$. Our groups $DL^{n+1}(\Lambda,P)$ may form part of a full classification of the doubly slice knots with $\pi_1(S^{n+2}\sm K)\cong\Z$. Note as well that although throughout Section \ref{sec:knots} we have worked with the coefficient rings $\Z[\Z]$ and $\Lambda$, the algebraic framework we have developed in this paper is robust enough to handle non-abelian fundamental groups, which we hope will be the topic of future work.

\item One might suppose that the doubly slice obstructions seen by the Blanchfield form over $\Z[\Z]$ encompass all abelian homological obstructions and that the use of the Blanchfield \emph{complex} over $\Z[\Z]$ is redundant as it does not use the fundamental group $\pi_1(S^{n+2}\sm K)$ (after all, this is the case in the slice problem). But in fact, abelian homology-level secondary obstructions were identified by Levine \cite[p.\ 252]{MR718271}. These homology-level obstructions involve the ring structure in cohomology. Product structures like this are not well accounted for in $L$--theory and are not seen by a class in double $L$--theory.\end{flushenumerate}\end{remark}

\bibliographystyle{gtart}
\def\MR#1{}
\bibliography{writeup}

\end{document}